\documentclass{siamart190516}

\usepackage{lipsum}
\usepackage{amsfonts}
\usepackage{graphicx}
\usepackage{epstopdf}
\usepackage{algorithmic}
\ifpdf
  \DeclareGraphicsExtensions{.eps,.pdf,.png,.jpg}
\else
  \DeclareGraphicsExtensions{.eps}
\fi

\newsiamremark{hypothesis}{Hypothesis}
\crefname{hypothesis}{Hypothesis}{Hypotheses}
\newsiamthm{claim}{Claim}

\headers{Least-squares neural network method}{Z. Cai, J. Chen AND M. Liu}

\title{Least-Squares Neural Network (LSNN) Method \\ for Scalar Nonlinear Hyperbolic Conservation Laws:\\ [1mm]
Discrete Divergence Operator\thanks{This work was supported in part by the National Science Foundation
under grant DMS-2110571.}}

\author{Zhiqiang Cai\thanks{Department of Mathematics, Purdue University, 150 N. University Street, West Lafayette, IN 47907-2067 
  (\email{caiz@purdue.edu}, \email{chen2042@purdue.edu}).}
\and Jingshuang Chen\footnotemark[2]
\and Min Liu\thanks{School of Mechanical Engineering, Purdue University, 585 Purdue Mall,
West Lafayette, IN 47907-2088(\email{liu66@purdue.edu}). }}

\usepackage{amsopn}

\ifpdf
\hypersetup{
  pdftitle={FOSLS},
  pdfauthor={Z.Cai, J.Chen, M.Liu}
}
\fi

\usepackage[utf8]{inputenc}
\usepackage{bm}
\usepackage{mathtools}
\usepackage{indentfirst}
\usepackage{subfigure}
\usepackage{tcolorbox}
\usepackage{color}
\usepackage[export]{adjustbox}
\usepackage{multirow}
\usepackage{scalerel}
\usepackage{booktabs}

\usepackage{algorithmic}
\Crefname{ALC@unique}{Line}{Lines}

\newcommand{\R}{\mathbb{R}}

\usepackage{graphicx}
\usepackage{diagbox}
\usepackage{pifont}
\newcommand{\vertiii}[1]{{\left\vert\kern-0.25ex\left\vert\kern-0.25ex\left\vert #1 
    \right\vert\kern-0.25ex\right\vert\kern-0.25ex\right\vert}}
\renewcommand{\div}{\nabla \cdot}
\newcommand{\bsigma}{\mbox{\boldmath${\sigma}$}}

\newcommand{\btheta}{\mbox{\boldmath${\theta}$}}

\newcommand{\bomega}{\mbox{\boldmath$\omega$}}

\newcommand{\btau}{\mbox{\boldmath${\tau}$}}

\newenvironment{myproof}[2] {\paragraph{Proof of {#1} {#2}}}{\hfill$\square$}

\newcommand{\jump}[1]{[\![ #1]\!]}

\usepackage{enumitem}
\setlist[itemize]{left=16pt} 

\def\bb{{\bf b}}

\def\bff{{\bf f}}

\def\bm{{\bf m}}
\def\bn{{\bf n}}
\def\bx{{\bf x}}

\def\bzz{{\bf z}}

\def\cL{{\cal L}}
\def\cM{{\cal M}}
\def\cN{{\cal N}}

\def\cT{{\cal T}}
\def\cV{{\cal V}}
\def\cK{{\cal K}}

\def\div{{\mbox{\bf div\,}}}
\def\divt{${\mbox{\bf div}}_{_\cT}$\,}

\begin{document}

\maketitle

\begin{abstract}
A least-squares neural network (LSNN) method was introduced for solving scalar linear and nonlinear hyperbolic conservation laws (HCLs) in \cite{Cai2021linear, Cai2021nonlinear}.
This method is based on an equivalent least-squares (LS) formulation and uses ReLU neural network as approximating functions, making it ideal for approximating discontinuous functions with unknown interface location. In the design of the LSNN method for HCLs, the numerical approximation of differential operators is a critical factor, and standard numerical or automatic differentiation along coordinate directions can often lead to a failed NN-based method. To overcome this challenge, this paper rewrites HCLs in their divergence form of space and time and introduces a new discrete divergence operator. As a result, the proposed LSNN method is free of penalization of artificial viscosity. 

Theoretically, the accuracy of the discrete divergence operator is estimated even for discontinuous solutions. Numerically, the LSNN method with the new discrete divergence operator was tested for several benchmark problems with both convex and non-convex fluxes, and was able to compute the correct physical solution for problems with rarefaction, shock or compound waves. The method is capable of capturing the shock of the underlying problem without oscillation or smearing, even without any penalization of the entropy condition, total variation, and/or artificial viscosity. 
\end{abstract}



\section{Introduction}\label{sec1}

Numerically approximating solutions of nonlinear hyperbolic conservation laws (HCLs) is a computationally challenging task. This is partly due to the discontinuous nature of HCL solutions at unknown locations, which makes approximation using fixed, quasi-uniform meshes very difficult. Over the past five decades, many advanced numerical methods have been developed to address this issue, including higher order finite volume/difference methods using limiters, filters, ENO/WENO, etc.(e.g.,  \cite{roe1981approximate, shu1988efficient, shu1998essentially,gottlieb1997gibbs,hesthaven2017numerical,hesthaven2007nodal, leveque1992numerical}) and discontinuous and/or adaptive finite element methods (e.g., \cite{CockburnShu1989, brezzi2004discontinuous,dahmen2012adaptive, demkowicz2010class, burman2009posteriori, houston1999posteriori, houston2000posteriori}). 

Neural networks (NNs) as a new class of approximating functions have been used recently for solving partial differential equations (see, e.g., \cite{cai2020deep,raissi2019physics,Sirignano18}) due to their versatile expressive power. One of the unique features of NNs is their ability to generate moving meshes implicitly by neurons that can automatically adapt to the target function and the solution of a PDE, which helps overcome the limitations of traditional approximation methods that use fixed meshes. For example, a ReLU NN generates continuous piece-wise linear functions with irregular and free/moving meshes. This property of ReLU NNs was used in \cite{Cai2021linear} for solving linear advection-reaction problem with discontinuous solution, without requiring information about the location of discontinuous interfaces. Specifically, the least-squares NN method studied in \cite{Cai2021linear} is based on the least-squares formulation in (\cite{bochev2001improved, de2004least}), and it uses ReLU NNs as the approximating functions while approximating the differential operator by directional numerical differentiation. Compared to various adaptive mesh refinement (AMR) methods that locate discontinuous interfaces through an adaptive mesh refinement process, the LSNN method is significant more efficient in terms of the number of degrees of freedom (DoF) used.


Solutions to nonlinear hyperbolic conservation laws are often discontinuous due to shock formation. It is well-known that the differential form of a HCL is not valid at shock waves, where the solution is discontinuous. As a result, the directional numerical differentiation of the differential operator based on the differential form used in \cite{Cai2021linear} cannot be applied to nonlinear HCLs. To overcome this challenge, the integral form of HCLs (as seen in \cite{leveque1992numerical}) must be used, which is valid for problems with discontinuous solutions, particularly at the discontinuous interfaces. This is why the integral form forms the basis of many conservative methods such as Roe's scheme \cite{harten1987uniformly}, WENO \cite{shu1998essentially, shu1988efficient}, etc.


Approximating the divergence operator by making use of the Roe and ENO fluxes, in \cite{Cai2021nonlinear} we tested the resulting LSNN method for scalar nonlinear HCLs. Numerical results for the inviscid Burgers equation showed that the LSNN method with conservative numerical differentiation is capable of capturing the shock without smearing and oscillation. Additionally, the LSNN method has fewer DoF than traditional mesh-based methods. Despite the promising results in \cite{Cai2021nonlinear}, limitations were observed with the LSNN method when using conservative numerical differentiation of the Roe and second-order ENO fluxes. For example, the resulting LSNN method is not accurate for complicated initial condition, and has problems with rarefaction waves and non-convex spatial fluxes. To improve accuracy, using ``higher order'' conservative methods such as ENO or WENO could be considered. However, these conservative schemes are designed for traditional mesh-based methods and the ``higher order'' here is measured at where solutions are smooth. 

In this paper, a new discrete divergence operator is proposed to accurately approximate the divergence of a vector filed even in the presence of discontinuities. This operator is defined based on its physical meaning: the rate of net outward flux per unit volume,  
and is approximated through surface integrals by the {\it composite} mid-point/trapezoidal numerical integration. Theoretically, the accuracy of the discrete divergence operator can be improved by increasing the number of surface integration points (as shown in Lemma~\ref{3.3} and Remark~\ref{r3.4}). The LSNN method, being a ``mesh/point-free'' space-time method, allows the use of all points on the boundary surfaces of a control volume for numerical integration.

Theoretically, we show that the residual of the LSNN approximation using the newly developed discrete divergence operator is bounded by the best approximation of the class of NN functions in some measure as stated in Lemma~\ref{4.1} plus the approximation error from numerical integration and differentiation (Lemma~\ref{4.3}). Numerically, our results show that the LSNN method with the new discrete divergence operator can accurately solve the inviscid Burgers equation with various initial conditions, compute the viscosity vanishing solution, capture shock without oscillation or smearing, and is much more accurate than the LSNN method in \cite{Cai2021nonlinear}. Note that the LSNN method does not use flux limiters. Moreover, the LSNN method using new discrete divergence operator works well for problems with non-convex flux and accurately simulates compound waves. 

Recently, several NN-based numerical methods have been introduced for solving scalar nonlinear hyperbolic conservation laws by various researchers (\cite{PNAS2019, cai21, Cai2021nonlinear, Cai2021linear, Fuks20, raissi2019physics, Patel22}). Those methods can be categorized as the physics informed neural network (PINN) \cite{PNAS2019, Fuks20, raissi2019physics, Patel22} and the least-squares neural network (LSNN) \cite{cai21, Cai2021nonlinear, Cai2021linear, cai2020deep} methods. First, both methods are based on the least-squares principle, but the PINN uses the discrete $l^2$ norm and the LSNN uses the continuous Sobolev norm depending on the underlying problem. Second, the differential operator of the underlying problem is approximated by either automatic differentiation or standard finite difference quotient for the PINN and by specially designed discrete differential operator for the LSNN. For example, the LSNN uses discrete directional differential operator in \cite{Cai2021linear} for linear advection-reaction problems, and various traditional conservative schemes in \cite{Cai2021nonlinear} or discrete divergence operator in this paper (see \cite{cai21} for its first version) for nonlinear scalar hyperbolic conservation laws. 

The original PINN has limitations that have been addressed in several studies (see, e.g., \cite{Fuks20, Patel22}). For nonlinear scalar hyperbolic conservation laws,  \cite{Fuks20} found that the PINN fails to provide reasonable approximate solution of the PDE and modified the loss function by penalizing the artificial viscosity term. \cite{Patel22} applied the discrete $l^2$ norm to the boundary integral equations over control volumes instead of the differential equations over points and modified the loss function by penalizing the entropy, total variation, and/or artificial viscosity. Even though the least-squares principle permits freedom of various penalizations, choosing proper penalization constants can be challenging in practice and it affects the accuracy,  efficiency, and stability of the method. In contrast, the LSNN does not require any penalization constants. 



The paper is organized as follows. Section~\ref{sec2} describes the hyperbolic conservation law, its least-squares formulation, and preliminaries. The space-time LSNN method and its block version are presented in Sections~\ref{sec3}. The discrete divergence operator and its error bound is introduced and analyzed in Section~\ref{sec4}. Finally, numerical results for various benchmark test problems are given in Section~\ref{sec5}.


\section{Problem Formulation}\label{sec2}
Let $\tilde{\Omega}$ be a bounded domain in ${\R}^d$ ($d=1, \,2$, or $3$) with Lipschitz boundary, and $I=(0, T)$ be the temporal interval.
Consider the scalar nonlinear hyperbolic conservation law 
\begin{equation} \label{pde}
    \left\{\begin{array}{rcll}
    u_t(\bx,t) + \nabla_{\bx} \!\cdot \tilde{\bff} (u) &= & 0, &\text{ in }\,\, \tilde{\Omega}  \times I, \\[2mm]
    u&=&
    \tilde{g}, &\text{ on }\,\, \tilde{\Gamma}_{-} ,\\[2mm]
    u(\bx,0) &=& u_0(\bx), &\text{ in }\,\, \tilde{\Omega},
    \end{array}\right.
\end{equation}
where $u_t$ is the partial derivative of $u$ with respect to the temporal variable $t$; $\nabla_{\bx}\cdot$ is a divergence operator with respect to the spatial variable $\bx$; $\tilde{\bff}(u)=(f_1(u),...,f_d(u))$ is the spatial flux vector field; $\tilde{\Gamma}_-$ is the part of the boundary $\partial \tilde{\Omega} \times I$ where the characteristic curves enter the domain $\tilde{\Omega}  \times I$; and the boundary data $\tilde{g}$ and the initial data $u_0$ are given scalar-valued functions. Without loss of generality, assume that $f_i(u)$ is twice differentiable for $i=1,\cdots, d$.


Problem (\ref{pde}) is a hyperbolic partial differential equation defined on a space-time domain $\Omega =\tilde{\Omega}  \times I$ in $\R^{d+1}$. Denote the inflow boundary of the domain $\Omega$ and the inflow boundary condition by 
 \[
 \Gamma_-=\left\{\begin{array}{ll}
  \tilde{\Gamma}_-, & t\in (0,T),\\[2mm]
  \Omega, & t=0
  \end{array}\right.
  \quad\mbox{and}\quad
  g=\left\{\begin{array}{ll}
  \tilde{g}, & \mbox{on }\tilde{\Gamma}_-, \\[2mm]
  u_0(\bx), & \mbox{on }\Omega,
  \end{array}\right.
 \]
respectively. Then (\ref{pde}) may be rewritten as the following compact form
\begin{equation} \label{pde1}
    \left\{\begin{array}{rccl}
    \div {\bff} (u) &= & 0, &\text{ in }\,\, {\Omega}\in \R^{d+1}, \\[2mm]
    u&=&
    {g},&\text{ on }\,\, {\Gamma}_{-} ,
    \end{array}\right.
\end{equation}
where $\mbox{\bf div}= (\partial_{x_1},\cdots,\partial_{x_d}, \partial_t)$ is a divergence operator with respect to both spatial and temporal variables $\bzz=(\bx,t)$, and $\bff (u) = (f_1(u),...,f_d(u), u)= (\tilde{\bff}(u),u)$ is the spatial and temporal flux vector field. 
Assume that $u\in L^\infty(\Omega)$. Then $u$ is called a weak solution of (\ref{pde1}) if and only if
\begin{equation}\label{weak}
-(\bff (u), \nabla \varphi)_{0,\Omega} + (\bn\cdot \bff (u),\varphi)_{0, \Gamma_-} = 0,\quad\forall\,\, \varphi \in C^1_{\Gamma_+}(\bar{\omega}),
\end{equation}
where $\Gamma_+=\partial\Omega\setminus\Gamma_-$ is the outflow boundary and $C^1_{\Gamma_+}(\bar{\omega})=\{\varphi\in C^1(\bar{\omega}) :\, \varphi=0 \mbox{ on } {\Gamma_+}\}$.

Denote the collection of square integrable vector fields whose divergence is also square integrable by
\[
H(\mbox{div};\Omega)= \left\{\btau\in L^2(\Omega)^{d+1} |\, \div\btau \in L^2(\Omega)\right\}.
\]
It is then easy to see that solutions of (\ref{pde1}) are in the following subset of $L^2(\Omega)$ 
 \begin{equation}\label{space}
     \cV_{\bff}=\left\{v\in L^2(\Omega) |\, \bff(v)\in H(\mbox{div};\Omega)\right\}.
 \end{equation}
Define the least-squares (LS) functional
 \begin{equation}\label{ls}
    \mathcal{L}(v;{ g}) = \| \div \bff(v)\|_{0,\Omega}^2 +  \|v-g\|_{0, \Gamma_-}^2 ,
\end{equation}
where $\|\cdot\|_{0,S}$ denotes the standard $L^2(S)$ norm for $S=\Omega$ and $\Gamma_-$.
Now, the corresponding least-squares formulation is to seek $u\in V_{\bff}$ such that
\begin{equation}\label{minimization1}
    \mathcal{L}(u;{g}) = \min_{\small v\in V_{\bff}} \mathcal{L}(v;{g}).
\end{equation}

\begin{proposition}
Assume that $u\in L^\infty(\Omega)$ is a piece-wise $C^1$ function. Then $u$ is a weak solution of {\em (\ref{pde1})} if and only if $u$ is a solution of the minimization problem in {\em (\ref{minimization1})}.
\end{proposition}
\begin{proof}
The proposition is a direct consequence of Theorem 2.5 in \cite{de2005numerical}.
\end{proof}






\section{Least-Squares Neural Network Method}\label{sec3}

Based on the least-squares formulation in (\ref{minimization1}), in this section we first describe the least-squares neural network (LSNN) method for the scalar nonlinear hyperbolic conservation law and then estimate upper bound of the LSNN approximation. 

To this end, denote a scalar-valued function generated by a $l$-layer fully connected neural network by
\begin{equation}\label{DNN1}
 \cN(\bzz)=\bomega^{(l)} \left(N^{(l-1)} \circ \cdots \circ N^{(2)}\circ N^{(1)}(\bzz)\right)- b^{(l)}:\, \bzz=(\bx,t)\in\R^{d+1}
\longrightarrow \R,
\end{equation}
where $\bomega^{(l)}\in \R^{n_{l-1}}$, $b^{(l)}\in \R$, and the symbol $\circ$ denotes the composition of functions. For $k=1,\,\cdots,\,l-1$, the $N^{(k)}: \R^{n_{k-1}} \rightarrow \R^{n_{k}}$ is called
the $k^{th}$ hidden layer of the network defined as follows:
\begin{equation}\label{layerdef}
  N^{(k)}(\bzz^{(k-1)})= \tau (\bomega^{(k)}\bzz^{(k-1)}-\bb^{(k)})
  \quad\mbox{for } \bzz^{(k-1)}\in \R^{n_{k-1}},
\end{equation}
where $\bomega^{(k)} 
\in \R^{n_{k}\times n_{k-1}}$, $\bb^{(k)}\in \R^{n_{k}}$, $\bzz^{(0)}=\bzz$, and $\tau(s)$ is the activation function whose application 
to a vector is defined component-wisely. In this paper, we will use the rectified linear unit (ReLU) activation function given by
\begin{equation}\label{tau-k}
 \tau(s) = \max\{0,\,s\}
 =\left\{\begin{array}{rclll}
 0, & \mbox{if }  s\leq 0,\\[2mm]
 s, & \mbox{if } s >0.
 \end{array}\right.
 \end{equation}
As shown in \cite{Cai2021linear}, the ReLU is a desired activation function for approximating discontinuous solution.
 
Denote the set of neural network functions by
\[
\cM_N=\cM_N(l)=\big\{\cN(\bzz) \mbox{ defined in (\ref{DNN1}) } :\,  \bomega^{(k)} 
\in \R^{n_{k}\times n_{k-1}},\,\, \bb^{(k)}\in \R^{n_{k}} \mbox{ for } k=1,\cdots,l
\big\},
\]
where the subscript $N$ denotes the total number of parameters ${\small\btheta}=\left\{\bomega^{(k)}, \bb^{(k)}\right\}$ given by
\[
N=M_d(l) =\sum^l_{k=1} n_{k}\times (n_{k-1}+1).
\]
Obviously, the continuity of the activation function $\tau(s)$ implies that $\cM_N$ is a subset of $C^0(\Omega)$. Together with the smoothness assumption on spatial flux $\tilde{\bff}(u)$, it is easy to see that $\cM_N$ is also a subset of $\cV_{\bff}$ defined in (\ref{space}). 

Since $\cM_N$ is not a linear subspace, it is then natural to discretize the HCL using a least-squares minimization formulation. Before defining the computationally feasible least-squares neural network (LSNN) method, let us first consider an intermediate least-squares neural network approximation: finding $u^{_N}(\bzz;{\small\btheta}^*) \in \cM_N$ such that 
\begin{equation}\label{L-NN}
     \mathcal{L}\big(u^{_N}(\cdot;{\small\btheta}^*);g\big)
     = \min\limits_{v\in \cM_N} \mathcal{L}\big(v(\cdot;{\small\btheta});{g}\big)
     = \min_{{\scriptsize \btheta}\in\R^{N}}\mathcal{L}\big(v(\cdot;{\small\btheta});{g}\big).
\end{equation}



\begin{lemma}\label{4.1}
Let $u$ be the solution of {\em (\ref{pde1})}, and let $u^{_N}\in \cM_N$ be a solution of {\em (\ref{L-NN})}. Assume that $\bff$ is twice differentiable, then there exists a positive constant $C$ such that 
\begin{equation}\label{L-residual}
    \begin{split}
        \mathcal{L}\big(u^{_N};g\big)
&=\inf_{v\in \cM_N} 
\left( \|v-u\|_{0, \Gamma_-}^2 + \big\| \div \left[\bff(v)- \bff(u)\right]\big\|_{0,\Omega}^2\right)\\[2mm]
&\leq  C\inf_{v\in \cM_N} 
\left( \|v-u\|_{0, \Gamma_-}^2 + \big\| \div \left[\bff^\prime(u) (v-u)\right]\big\|_{0,\Omega}^2\right) + \mbox{h.o.t.},
    \end{split}
\end{equation}
where h.o.t. means a higher order term comparing to the first term.
\end{lemma}

\begin{proof}
For any $v\in \cM_N$, (\ref{L-NN}) and (\ref{pde1}) imply that
\[
\mathcal{L}\big(u^{_N};g\big)\leq \mathcal{L}\big(v;g\big) =\|v-u\|_{0, \Gamma_-}^2+\big\| \div \left[\bff(v)-\bff(u)\right]\big\|_{0,\Omega}^2,
\]
which proves the validity of the equality in (\ref{L-residual}). 
By the Taylor expansion, there exists $\{w_i\}_{i=1}^d$ between $u$ and $v$ such that 
\[
\bff(v)-\bff(u)=\bff^\prime(u) (v-u) +\dfrac12 \bff^{\prime\prime}(w) (v-u)^2,
\]
where $\bff^\prime(u)=(f^{\prime}_1(u),\cdots,f^{\prime}_d(u), 1)^t$ and $\bff^{\prime\prime}(w)=(f^{\prime\prime}_1(w_1),\cdots,f^{\prime\prime}_d(w_d), 0)^t$.
Together with the triangle inequality we have
\begin{equation}\label{4.12}
\big\| \div \left[\bff(v)-\bff(u)\right]\big\|_{0,\Omega} \leq \big\| \div \left[\bff^\prime(u) (v-u)\right]\big\|_{0,\Omega} +  \dfrac12\,\left\| \div \left[\bff^{\prime\prime}(w) (v-u)^2\right]\right\|_{0,\Omega}.
\end{equation}
Notice that the second term in the right-hand side of (\ref{4.12}) is a higher order term comparing to the first term. Now, the inequality in (\ref{L-residual}) is a direct consequence of the equality in (\ref{L-residual}) and (\ref{4.12}).
This completes the proof of the lemma.
\end{proof}

\begin{remark}
When $u$ is sufficiently smooth, the second term 
\[
\div \left[\bff^\prime(u) (v-u)\right] = (v-u)\, \div \bff^\prime(u) + \bff^\prime(u) \!\cdot \!\nabla (v-u)
\]
may be bounded by the sum of the $L^2$ norms of $v-u$ and the directional derivative of $v-u$ along the direction $\bff^\prime(u)$.
\end{remark}

Evaluation of the least-squares functional $\mathcal{L}\big(v;\,g\big)$ defined in (\ref{ls}) requires integration and differentiation over the computational domain and the inflow boundary. As in 
\cite{cai2020deep}, we evaluate the integral of the least-squares functional by numerical integration. 
To do so, let 
\[
{\cal T}=\{K :\, K\mbox{ is an open subdomain of } \Omega \} \quad\mbox{and}\quad
{\cal E}_{-}=\{E
=\partial K \cap \Gamma_-:\,\, K\in\mathcal{T}\}
\] 
be partitions of the domain $\Omega$ and the inflow boundary $\Gamma_-$, respectively. For each $K\in {\cal T}$ and $E\in {\cal E}_-$, let ${\cal Q}_K$ and ${\cal Q}_E$ be Newton-Cotes quadrature of integrals over $K$ and $E$, respectively. 
The corresponding discrete least-squares functional is defined by
\begin{equation}\label{L-NN-d}
\begin{split}
     \mathcal{L}_{_{\small {\cal T}}}\big(v;{g}\big) 
     = \sum_{K \in {\cal T}} {\cal Q}_K^2\big(\div_{\!\!_\cT} \bff(v) \big)+  \sum\limits_{E\in {\cal E}_-} {\cal Q}_E^2\big(v-g\big),
\end{split}
\end{equation}
where 
\divt denotes a discrete divergence operator. The discrete divergence operators of the Roe and ENO type were studied in \cite{Cai2021nonlinear}. In the subsequent section, we will introduce new discrete divergence operators tailor to the LSNN method that are accurate approximations to the divergence operator when applying to discontinuous solution. 

With the discrete least-squares functional $\mathcal{L}_{_{\small {\cal T}}}\big(v;\,{g}\big)$,
the least-squares neural network (LSNN) method is to find ${u}^{_N}_{_{\small {\cal T}}}(\bzz,{\small\btheta}^*)\in \cM_N$ such that
 \begin{equation}\label{discrete_minimization_functional}
  \mathcal{L}_{_{\small {\cal T}}} \big({u}^{_N}_{_{\small {\cal T}}}(\cdot,{\small\btheta}^*);{g}\big) 
  = \min\limits_{v\in \cM_N} \mathcal{L}_{_{\small {\cal T}}}\big(v(\cdot;{\small\btheta});\,{g}\big)
 = \min_{{\scriptsize \btheta}\in\R^{N}}\mathcal{L}_{_{\small {\cal T}}} \big(v(\cdot; {\small\btheta});{g}\big).
\end{equation}


\begin{lemma}\label{4.3}
Let $u$, $u^{_N}$, and $u_{_\cT}^{_N}$ be the solutions of problems {\em (\ref{ls})}, {\em (\ref{L-NN})}, and {\em (\ref{discrete_minimization_functional})}, respectively. Then we have
\begin{equation}\label{Cea-L-d}
   \mathcal{L}\big({u}^{_N}_{_{\small {\cal T}}};{g}\big) 
    \le \Big| \big(\mathcal{L}-\mathcal{L}_{_{\small {\cal T}}}\big)\big({u}^{_N}_{_{\small {\cal T}}};{g}\big)\Big| + \Big| \big(\mathcal{L}-\mathcal{L}_{_{\small {\cal T}}}\big)\big({u}^{_N};{g}\big)\Big| +\Big| \mathcal{L}\big(u^{_N};{g}\big)\Big|.
\end{equation}
\end{lemma}

\begin{proof}
By the fact that $\cL_{_\cT}(u_{_\cT}^{_N}; {\bf f}) \leq \cL_{_\cT}(u^{_N}; {\bf f})$, we have
\begin{eqnarray}\nonumber
\mathcal{L}\big({u}^{_N}_{_{\small {\cal T}}};{g}\big)
&=& \big(\mathcal{L}-\mathcal{L}_{_{\small {\cal T}}}\big)\big({u}^{_N}_{_{\small {\cal T}}};{g}\big) + \mathcal{L}_{_{\small {\cal T}}}\big({u}^{_N}_{_{\small {\cal T}}};{g}\big)
\leq \big(\mathcal{L}-\mathcal{L}_{_{\small {\cal T}}}\big)\big({u}^{_N}_{_{\small {\cal T}}};{g}\big) + \mathcal{L}_{_{\small {\cal T}}}\big({u}^{_N};{g}\big)\\[2mm] \label{4.10a}
&=& \big(\mathcal{L}-\mathcal{L}_{_{\small {\cal T}}}\big)\big({u}^{_N}_{_{\small {\cal T}}};{g}\big) + \big(\mathcal{L}_{_{\small {\cal T}}}-\mathcal{L}\big)\big({u}^{_N};{g}\big) + \mathcal{L}\big({u}^{_N};{g}\big),
\end{eqnarray}
which, together with the triangle inequality, implies 
(\ref{Cea-L-d}).
\end{proof}

This lemma indicates that the minimum of the discrete least-squares functional $\mathcal{L}_{_{\small {\cal T}}}$ over $\cM_N$ is bounded by the minimum of the least-squares functional $\mathcal{L}$ over $\cM_N$ plus the approximation error of numerical integration and differentiation in $\cM_N$.

In the remainder of this section, we describe the block space-time LSNN method introduced in \cite{Cai2021nonlinear} for dealing with the training difficulty over a relative large computational domain $\Omega$. The method is based on a partition $\{\Omega_{k-1,k}\}_{k=1}^{n_b}$ of the computational domain $\Omega$. To define $\Omega_{k-1,k}$, let $\{\Omega_k\}_{k=1}^{n_b}$ be subdomains of $\Omega$ satisfying the following inclusion relation
 \[
 \emptyset=\Omega_0\subset\Omega_1\subset \cdots \subset \Omega_{n_b}=\Omega.
 \]
Then set $\Omega_{k-1,k}=\Omega_k\setminus \Omega_{k-1}$ for $k=1,\cdots, n_b$. Assume that $\Omega_{k-1,k}$ is in the range of influence of 
\[
\Gamma_{k-1,k}=\partial \Omega_{k-1,k}\cap \partial\Omega_{k-1} \quad\mbox{and}\quad
\Gamma^k_-=\partial \Omega_{k-1,k}\cap \Gamma_-.
\]

Denote by $u^k=u|_{\Omega_{k-1,k}}$ the restriction of the solution $u$ of (\ref{pde1}) on $\Omega_{k-1,k}$, then $u^k$ is the solution of the following problem:
\begin{equation} \label{pde2}
    \left\{\begin{array}{rccl}
    \div_{\!\!_\cT} {\bff} (u^k) &= & 0, &\text{ in }\,\, \Omega_{k-1,k}\in \R^{d+1}, \\[2mm]
    u^k&=&
    u^{k-1},&\text{ on }\,\, \Gamma_{k-1,k} , 
    \\[2mm]
    u^k&=&
    {g},&\text{ on }\,\, \Gamma^k_-.
    \end{array}\right.
\end{equation}
Let
 \[
 \mathcal{L}^k\big(v;u^{k-1},g\big) = 
 \| \div \bff(v)\|_{0,\Omega_{k-1,k}}^2 +
 \|v-u^{k-1}\|_{0, \Gamma_{k-1,k}}^2+\|v-g\|_{0, \Gamma^k_-}^2,
 \]
and define the corresponding discrete least-squares functional
$\mathcal{L}^k_{_{\small {\cal T}}}\big(v;u^{k-1},g\big)$ over the subdomain $\Omega_{k-1,k}$ in a similar fashion as in (\ref{L-NN-d}). Now, the block space-time LSNN method is to find 
${u}^k_{_{\small {\cal T}}}(\bzz,{\small\btheta}^*_k)\in \cM_N$ such that
 \begin{equation}\label{discrete_minimization_functional-block}
  \mathcal{L}^k_{_{\small {\cal T}}} \big({u}^k_{_{\small {\cal T}}}(\cdot,{\small\btheta}_k^*);u^{k-1},{g}\big) 
  = \min\limits_{v\in \cM_N} \mathcal{L}^k_{_{\small {\cal T}}}\big(v(\cdot;{\small\btheta});\,u^{k-1},{g}\big)
 = \min_{{\scriptsize \btheta}\in\R^{N}}\mathcal{L}^k_{_{\small {\cal T}}} \big(v(\cdot; {\small\btheta});u^{k-1},{g}\big)
\end{equation}
for $k=1,\cdots,n_b$.

\section{Discrete Divergence Operator}\label{sec4}

As seen in \cite{Cai2021linear, Cai2021nonlinear}, numerical approximation of the differential operator is critical for the success of the LSNN method. Standard numerical or automatic differentiation along coordinate directions generally results in an inaccurate LSNN method, even for linear problems when solutions are discontinuous. This is because the differential form of the HCL is invalid at discontinuous interface. To overcome this difficulty, we used the discrete directional differentiation for linear problems in \cite{Cai2021linear} and the discrete divergence operator of the Roe and ENO type for nonlinear problems in \cite{Cai2021nonlinear}. 



In this section, we introduce a new discrete divergence operator based on the definition of the divergence operator. Specifically, for each $K\in \cT$, let $\bzz^i_{_K}=(\bx^i_{_K}, t^i_{_K})$ and $\omega_i$ for $i\in J$ be the quadrature points and weights for the quadrature ${\cal Q}_K$, where $J$ is the index set. Hence, the discrete least-squares functional becomes
\[
\mathcal{L}_{_{\small {\cal T}}}\big(v;{g}\big) 
     = \sum_{K \in {\cal T}} \left(\sum_{i\in J} \omega_i \,\div_{\!\!_\cT} \bff\big(v(\bzz^i_{_K})\big)\right)^2+  \sum\limits_{E\in {\cal E}_-} {\cal Q}_E^2\big(v-g\big).
\]
To define the discrete divergence operator $\div_{\!\!_\cT}$, we first construct a set of control volumes 
\[
\cV=\{V :\, V\mbox{ is an open subdomain of } \Omega \}
\]
such that $\cV$ is a partition of the domain $\Omega$ and that each quadrature point is the centroid of a control volume $V\in \cV$. Denote by $V^i_{_K}$ the control volume corresponding to the quadrature point $\bzz^i_{_K}$, by the definition of the divergence operator, we have
\begin{equation}\label{DO}
    \div \bff\big(u(\bzz^i_{_K})\big) \approx \mbox{avg}_{V^i_{_K}}\div \bff(u)=
 \dfrac{1}{|V^i_{_K}|}\int_{\partial V^i_{_K}} \bff(u)\cdot\bn\,dS,
\end{equation}
where the average of a function $\varphi$ over $V^i_{_K}$ is defined by
\[
\mbox{avg}_{V^i_{_K}}\varphi=\dfrac{1}{|V^i_{_K}|} \int_{V^i_{_K}}\varphi(\bzz)\, d\bzz. 
\]
The average of $\varphi$ with respect to the partition $\cV$ is denoted by $\mbox{avg}_{_\cV}\varphi$ and
defined as a piece-wise constant function through its restriction on each $V\in\cV$ by
\[
\mbox{avg}_{_\cV}\varphi\big|_V=\mbox{avg}_V\varphi. 
\]
Now we may design a discrete divergence operator \divt acting on the total flux $\bff(u)$ by approximating the surface integral on the right-hand side of (\ref{DO}). 

All existing conservative schemes of various order such as Roe, ENO, WENO, etc. may be viewed as approximations of the surface integral using values of $\bff(u)$ at some {\it mesh points}, where most of them are outside of $\bar{V}$. These conservative schemes are nonlinear methods because the procedure determining proper mesh points to be used for approximating the average of the spatial flux is a nonlinear process due to possible discontinuity.

Because the LSNN method is a ``mesh/point-less'' space-time method, all points on $\partial V\in \R^{d+1}$ are at our disposal for approximating the surface integral. Hence, the surface integral can be approximated as accurately as desired by using only points on $\partial V$. When $u$ and hence $f_i(u)$ are discontinuous on $\partial V$, the best linear approximation strategy is to use piece-wise constant/linear functions on a sufficiently fine partition of each face of $\partial V$, instead of higher order polynomials on each face. This suggests that a composite lower-order numerical integration such as the composite mid-point/trapezoidal quadrature would provide accurate approximation to the surface integral in (\ref{DO}), and hence the resulting discrete divergence operator would be accurate approximation to the divergence operator, even if the solution is discontinuous. 


\subsection{One Dimension}

For clarity of presentation, the discrete divergence operator described above will be first introduced in this section in one dimension. To this end, to approximate single integral $I(\varphi)=\int^d_c\varphi(s)\,ds$, we will use the composite midpoint/trapezoidal rule:
\begin{equation}\label{integration}
    Q(\varphi(s);c,d,p)=\left\{\begin{array}{ll}
    \dfrac{d-c}{p}\sum\limits_{i=0}^{p-1}\varphi\big(s_{i+1/2}\big), & \mbox{midpoint},\\[6mm]
    \dfrac{d-c}{2p}\left(\varphi(c)+\varphi(d)+2\sum\limits_{i=1}^{p-1}\varphi\big(s_{i}\big)\right), & \mbox{trapezoidal},
    \end{array}\right.
\end{equation}
where $\{s_i\}_{i=0}^p$ uniformly partitions the interval $[c,d]$ into $p$ sub-intervals.

Let $\Omega=(a,b)\times (0,T)$. For simplicity, assume that the integration partition $\cT$ introduced in Section~\ref{sec3} is a uniform partition of the domain $\Omega$; i.e., 
\[
{\cal T}=\{K=K_{ij} :\, i=0,1, \cdots ,m-1;\,\, j=0,1, \cdots ,n-1\} \mbox{ with } K_{ij} = (x_{i},x_{i+1})\times (t_{j},t_{j+1}),
\] 
where $x_i=a+ih$ and $t_j=j\tau$ with $h=(b-a)/m$ and $\delta=T/n$. For integration subdomain $K_{ij}$, the set of quadrature points is
\[
\begin{array}{ll}
M_{ij}=\{\bzz_{i+\frac12,j+\frac12}\}     & \mbox{for the midpoint rule}, \\ 
 [2mm]
 T_{ij}=\{\bzz_{i,j}, \bzz_{i+1,j},\bzz_{i,j+1},\bzz_{i+1,j+1}\}     & \mbox{for the trapezoidal rule}, \\ [2mm]
 \mbox{and }\,\, S_{ij}=M_{ij}\cup T_{ij}\cup \{\bzz_{i+\frac12,j}, \bzz_{i,j+\frac12},\bzz_{i+1,j+\frac12},\bzz_{i+\frac12,j+1}\}     & \mbox{for the Simpson rule} ,
\end{array}
\]
where $\bzz_{i+k,j+l}=\big(x_i+kh,t_j+l\delta\big)$ for $k,l=0$, $1/2,$ or $1$.
Based on those quadrature points, the sets of control volumes may be defined accordingly. For example, the control volume $\cV_m$ for the midpoint rule is $\cT$; the control volume $\cV_t$ for the trapezoidal rule is obtained by shifting control volumes in $\cV_m$ by $\dfrac12 \, (h,\delta)$ plus half-size control volumes along the boundary; and the control volume $\cV_s$ for the Simpson rule is obtained in a similar fashion as $\cV_t$ on the element size of $h/2$ and $\delta/2$ for space and time, respectively.

For simplicity of presentation, we define the discrete divergence operator only for the midpoint rule for it can be defined in a similar fashion for other quadrature. Since $\cV_m=\cT$, i.e., the control volume of $\cV_m$ is the same as the element of $\cT$, for each control volume $V=K_{ij}$, denote its centroid by 
\[
\bzz_V=\bzz_{ij}=(x_i+h/2, t_j+\delta/2).
\]
Denote by $\sigma= f(u)$ the spatial flux, then the total flux is the two-dimensional vector field $\bff (u)= (\sigma,u)$. Denote the first-order finite difference quotients by
\[
\sigma(x_{i},x_{i+1};t)=\dfrac{\sigma(x_{i+1},t)-\sigma(x_{i},t)}{x_{i+1}-x_i}
\quad\mbox{and}\quad
u(x;t_{j},t_{j+1})=\dfrac{u(x,t_{j+1})-u(x,t_{j})}{t_{j+1}-t_{j}}.
\]
Then the surface integral in (\ref{DO}) becomes
\begin{eqnarray}\label{SI}
    \dfrac{1}{|K_{ij}|}\int_{\partial K_{ij}}\bff (u)\cdot\bn\,dS= \delta^{-1}\int^{t_{j+1}}_{t_{j}}\sigma(x_{i},x_{i+1};t)\,dt + h^{-1} \int^{x_{i+1}}_{x_{i}}u(x;t_{j},t_{j+1})\,dx.
\end{eqnarray}
Approximating single integrals by the composite midpoint/trapezoidal rule, we obtain the following discrete divergence operator
\begin{equation}\label{dDO}
    \div{\!\!_{_\cT}} {\bff\big(u(\bzz_{{ij}})\big)}=\delta^{-1} Q(\sigma(x_{i},x_{i+1};t);t_j,t_{j+1},\hat{n}) + h^{-1} Q(u(x;t_j,t_{j+1});x_{i},x_{i+1},\hat{m}).
\end{equation}

\begin{remark}
Denote by $u_{i,j}$ as approximation to $u(x_i,t_j)$. {\em (\ref{dDO})} with $\hat{m}=\hat{n}=1$ using the trapezoidal rule leads to the following implicit conservative scheme for the one-dimensional scalar nonlinear HCL:
\begin{equation}\label{FV}
    \dfrac{u_{i+1,j+1}+u_{i,j+1}}{\delta} 
 + \dfrac{f\big(u_{i+1,j+1}\big)-f\big(u_{i,j+1}\big)}{h}
 =  \dfrac{u_{i+1,j}+u_{i,j}}{\delta} 
 -\dfrac{f\big(u_{i+1,j}\big)-f\big(u_{i,j}\big)}{h}
\end{equation}
for $i=0,1, \cdots ,m-1$ and $j=0,1, \cdots , n-1$.
\end{remark}

Below, we state error estimates of the discrete divergence operator defined in (\ref{dDO}) and postpone their proof to Appendix.

\begin{lemma}\label{3.2}
For any $K_{ij}\in \cT$, assume that $u$ is a $C^2$ function on every edge of the rectangle $\partial K_{ij}$. Then there exists a constant $C>0$ such that
\begin{eqnarray}  \nonumber
 && \| \div_{\!\!{_\cT}} \bff(u)-{\emph{\text{avg}}}_{{_\cT}}
  \div \bff (u)\|_{L^p(K_{ij})}\\[2mm] \label{div-est1_0}
& \leq &C\left(\dfrac{h^{1/p}\delta^{2}}{\hat{n}^2}\|\sigma_{tt}(x_{i+1},x_{i};\cdot)\|_{L^p(t_j,t_{j+1})} 
    + \dfrac{h^{2}\delta^{1/p}}{\hat{m}^2}\|u_{xx}(\cdot;t_{j+1},t_{j})\|_{L^p(x_i,x_{i+1})}\right).
\end{eqnarray}
\end{lemma}

This lemma indicates that $\hat{m}=1$ and $\hat{n}=1$ are sufficient if the solution is smooth on $\partial K_{ij}$. In this case, we may use higher order numerical integration, e.g., the Gauss quadrature, to approximate the surface integral in (\ref{SI}) for constructing a higher order discrete divergence operator.

When $u$ is discontinuous on $\partial K_{ij}$, error estimate on the discrete divergence operator becomes more involved. To this end, first we consider the case that the discontinuous interface $\Gamma_{ij}$ (a straight line) intersects two horizontal boundary edges of $K_{ij}$. 
Denote by $u_{ij}=u|_{K_{ij}}$ the restriction of $u$ in $K_{ij}$ and by $\jump{u_{ij}}_{t_l}$ the jump of $u_{ij}$ on the horizontal boundary edge $t=t_l$ of $K_{ij}$, where $l=j$ and $l=j+1$.

\begin{lemma}\label{3.3}
Assume that $u$ is a $C^2$ function of $t$ and a piece-wise $C^2$ function of $x$ on two vertical and two horizontal edges of $K_{ij}$, respectively. Moreover, $u$ has only one discontinuous point on each horizontal edge. 
Then there exists a constant $C>0$ such that
\begin{eqnarray} \nonumber
  &&   \| \div_{\!\!_\cT}  \bff(u)-{\emph{\text{avg}}}_{_\cT}\div \bff (u) \|_{L^p(K_{ij})} \\[2mm] \label{div-est2_0}  
     &\leq & C\left(\dfrac{h^{1/p}\delta^{2}}{\hat{n}^2}
    + \dfrac{h^{2}\delta^{1/p}}{\hat{m}^2} +\dfrac{h\delta^{1/p}}{\hat{m}^{1+1/q}}\right) 
    + \dfrac{(h\delta)^{1/p}}{\hat{m}}\sum_{l=j}^{j+1} \jump{u_{ij}}_{t_l}.
\end{eqnarray}
\end{lemma}

\begin{remark}\label{r3.4}
{\em Lemma~\ref{3.3}} implies that the choice of the number of sub-intervals of $(x_i,x_{i+1})$ on the composite numerical integration depends on the size of the jump of the solution and that large $\hat{m}$ would guarantee accuracy of the discrete divergence operator when $u$ is discontinuous on $\partial K_{ij}$.
\end{remark}

\begin{remark}\label{4.5}
Error bounds similar to {\em (\ref{div-est2_0})} hold for the other cases: $\Gamma_{ij}$ intercepts {\em (i)} two vertical edges or {\em (ii)} one horizontal and one vertical edges of $K_{ij}$. Specifically, we have
\[
\| \div_{\!\!_\cT}  \bff(u)-{\emph{\text{avg}}}_{_\cT}\div \bff (u) \|_{L^p(K_{ij})} \leq  C\left(\dfrac{h^{1/p}\delta^{2}}{\hat{n}^2}
    + \dfrac{h^{2}\delta^{1/p}}{\hat{m}^2} +\dfrac{h^{1/p}\delta }{\hat{n}^{1+1/q}}\right) 
    + \dfrac{(h\delta)^{1/p}}{\hat{n}}\sum_{l=i}^{i+1} \jump{\sigma_{ij}}_{x_l}
\]
for the case {\em (i)} and 
\[
\| \div_{\!\!_\cT}  \bff(u)-{\emph{\text{avg}}}_{_\cT}\div \bff (u) \|_{L^p(K_{ij})} \leq  C\left(\dfrac{h^{1/p}\delta^{2}}{\hat{n}^2}
    + \dfrac{h^{2}\delta^{1/p}}{\hat{m}^2}
    +\dfrac{h\delta^{1/p}}{\hat{m}^{1+1/q}}+\dfrac{h^{1/p}\delta }{\hat{n}^{1+1/q}}\right) 
    +E_{ij}
\]
for the case {\em (ii)}, where $E_{ij}=(h\delta)^{1/p}\left(\dfrac{1}{\hat{m}}\jump{u_{ij}}_{t_l} + \dfrac{1}{\hat{n}}\jump{\sigma_{ij}}_{x_l}\right)$ with $x_l=x_i$ or $x_{i+1}$ and $t_l=t_j$ or $t_{j+1}$.
\end{remark}

\subsection{Two Dimensions}
This section describes the discrete divergence operator in two dimensions. As in one dimension, the discrete divergence operator is defined as an approximation to the average of the divergence operator through the composite mid-point/trapezoidal quadrature to approximate the surface integral (\ref{DO}). Extension to three dimensions is straightforward. 

To this end, we first describe the composite mid-point/trapezoidal numerical integration for approximating a double integral over a rectangle region $T=(c_1,d_1)\times (c_2,d_2)$
\begin{eqnarray*}
  I(\varphi)&=&\int_T \varphi(s_1,s_2)\,ds_1ds_2\\[2mm] &\approx & Q\big(\varphi(s_1,s_2);c_1,d_1,p_1;c_2,d_2,p_2\big)\equiv Q\Big(Q\big(\varphi(s_1,\cdot);c_1,d_1,p_1\big)(s_2);c_2,d_2,p_2\Big),
\end{eqnarray*}
where $Q\big(\varphi(s_1,\cdot);c_1,d_1,p_1\big)$ is the composite quadrature defined in (\ref{integration}).

For simplicity, let $\Omega=\tilde{\Omega}\times I = (a_1,b_1)\times (a_2,b_2)\times (0,T)$, and assume that the integration partition $\cT$ introduced in Section~\ref{sec3} is a uniform partition of the domain $\Omega$; i.e., 
\[
{\cal T}=\{K=K_{ijk} :\, i=0,1, \cdots ,m_1-1;\,\, j=0,1, \cdots ,m_2-1; \,\, k=0,1, \cdots ,n-1\}
\] 
with $K_{ijk} = (x_{i},x_{i+1})\times (y_{j},y_{j+1})\times (t_{k},t_{k+1})$, 
where 
\[
x_i=a_1+ih_1, \quad y_j=a_2+jh_2, \quad\mbox{and}\quad t_k=k\delta,
\]
and $h_l=(b_l-a_l)/m_l$ for $l=1,\,2$ and $\delta=T/n$ are the respective spatial and temporal sizes of the integration mesh. Again, we define the discrete divergence operator only corresponding to the midpoint rule. Denote the mid-point of $K_{ijk}$ by
\[ 
\bzz_{ijk}=(x_i+\dfrac{h_1}{2}, y_j+\dfrac{h_2}{2}, t_k+\dfrac{\delta}{2}).
\] 

Let $\bsigma=(\sigma_1,\sigma_2)= (f_1(u),f_2(u))$, then the space-time flux is the three-dimensional vector field: $\bff (u)= (\bsigma,u)=(\sigma_1,\sigma_2,u)$. Denote the the first-order finite difference quotients by
\begin{eqnarray*}
&& \sigma_1(y,t;x_{i},x_{i+1})=\dfrac{\sigma_1(x_{i+1},y,t)-\sigma_1(x_{i},y,t)}{x_{i+1}-x_i}, \quad \sigma_2(x,t;y_{j},y_{j+1})=\dfrac{\sigma_2(x,y_{j+1},t)-\sigma_1(x_,y_j,t)}{y_{j+1}-y_{j}},\\[2mm]
&& \mbox{and }\,\,  u(x,y;t_{k},t_{k+1})=\dfrac{u(x,y,t_{k+1})-u(x,y,t_{k})}{t_{k+1}-t_{k}}.
\end{eqnarray*}
Denote three faces of $\partial K_{ijk}$ by
\[
K_{ij}^{xy}=(x_{i},x_{i+1})\times (y_{j},y_{j+1}), \,\, K_{ik}^{xt}=(x_{i},x_{i+1})\times (t_{k},t_{k+1}), \,\mbox{ and }\, K_{jk}^{yt}= (y_{j},y_{j+1})\times (t_{k},t_{k+1}).
\]
Then the surface integral in (\ref{DO}) becomes
\begin{eqnarray}\nonumber
    && \dfrac{1}{|K_{ijk}|}
    \int_{\partial K_{ijk}}\!\!\!\!\bff (u)\cdot\bn\,dS = (h_2\delta)^{-1}\int_{K^{yt}_{jk}}\!\!\sigma_1(y,t;x_{i+1},x_i)\,dydt \\[2mm] \label{SI-3}
    && \qquad\quad +\, (h_1\delta)^{-1}\int_{K^{xt}_{ik}}\!\!\sigma_2(x,t;y_{j+1},y_j)\,dxdt +(h_1h_2)^{-1} \int_{K^{xy}_{ij}}u(x,y;t_{k+1},t_k)\,dxdy.
\end{eqnarray}
Approximating double integrals by the composite midpoint/trapezoidal rule, we obtain the following discrete divergence operator
\begin{eqnarray}\nonumber 
     \div{\!\!_{_\cT}} {\bff\big(u(\bzz_{{ijk}})\big)}
    &=& (h_2\delta)^{-1} Q\big(\sigma_1(y,t;x_{i+1},x_i);y_j,y_{j+1},\hat{m}_2;t_k,t_{k+1},\hat{n}\big) \\[2mm] \nonumber 
    &&  + (h_1\delta)^{-1} Q\big(\sigma_2(x,t;y_{j+1},y_j);x_i,x_{i+1},\hat{m}_1;t_k,t_{k+1},\hat{n}\big)\\[2mm]\label{dDO-2}
    && 
    + (h_1h_2)^{-1} Q\big(u(x,y;t_{k+1},t_k);x_i,x_{i+1},\hat{m}_1;y_j,y_{j+1},\hat{m}_2\big).
\end{eqnarray}

\subsection{Integration mesh size} 
The discrete divergence operator defined in (\ref{dDO}) and (\ref{dDO-2}) for the respective one- and two- dimension is based on the composite midpoint/trapezoidal rule. As shown in Lemmas~\ref{3.2} and \ref{3.3} and Remark~\ref{4.5}, the discrete divergence operator can be as accurate as desired for the discontinuous solution provided that the size of integration mesh is sufficiently small. 

To reduce computational cost, note that the discontinuous interfaces of the solution $u$ lie on $d$-dimensional hyper-planes. Hence, they only intersect with a small portion of control volumes in $\cT$. This observation suggests that sufficiently fine meshes are only needed for control volumes at where the solution is possibly discontinuous. To realize this idea, we divide the set of control volumes into two subsets: 
\[
\cT=\cT_c\cup \cT_d, 
\]
where the solution $u$ is continuous in each control volume of $\cK^l_c$ and possibly discontinuous at some control volumes of $\cT_d$; i.e., 
\[
 \cT_c=\{K\in\cT :\, u\in C(K)\}
 \quad\mbox{and}\quad  \cT_d =\cT\setminus \cT_c.
\]

Next, we describe how to determine the set of control volumes $\cT_d$ in one dimension by the range of influence. It is well-known that characteristic curves are straight lines before their interception and are given by
\begin{equation}\label{charac-curve}
    x = x(T_l)+ \left(t-T_l\right)f^\prime\big(u\left(x(T_l),T_l\right)\big). 
\end{equation}
For $i=0,1, \cdots, m$, let
\[
\hat{x}_i=x_i+ \left(T_{l+1}-T_l\right)f^\prime\big(u^l_{_N}\left(x_i,T_l\right)\big),
\]
where $u^l_{_N}\left(x_i,T_l\right)$ is the neural network approximation from the previous time block \[
\Omega\times I_{l-1}=(a,b)\times(T_{l-1},T_l). 
\]

Clearly, the solution $u$ is discontinuous in a control volume $V_i\times I_l^{k}$ if either (1) $u(x, T_l)$ is discontinuous at the interval $V_i$ or (2) there are two characteristic lines intercepting in $V_i\times I_l^k$. In the first case, $V_i\times I_l^k$ is in $\cK^l_d$ if $u^l_{_N}(x,T_l)$ has a sharp change in the interval $V_i$; moreover, either $V_{i-1}\times I_l^k\in \cK^l_d$ if $\hat{x}_i< x_i$ or $V_{i+1}\times I_l^k\in \cK^l_d$ if $\hat{x}_{i+1}> x_{i+1}$. In the second case, assume that $\hat{x}_i>\hat{x}_{i+1}$, then $V_i\times I_l^k\in \cK^l_d$ if $\hat{x}_i<x_{i+1}$.

\section{Numerical Experiments}\label{sec5}
This section presents numerical results of the block space-time LSNN method for one and two dimensional problems. Let $\Omega = \tilde{\Omega}\times (0,T)$. The $k^{\text{th}}$ space-time block is defined as 
\[
\Omega_{k-1,k}=\Omega_{k}\setminus \Omega_{k-1}=\tilde{\Omega}\times \left(\frac{(k-1)T}{n_b}, \frac{kT}{n_b}\right) \quad\mbox{for }\,\, k=1,\cdots,n_b,
\] 
where $\Omega_k = \tilde{\Omega}\times \left(0, kT/n_b\right)$. For efficient training, the least-squares functional is modified as follows:
\begin{equation}\label{training_weight}
 \mathcal{L}^k\big(v;u^{k-1},g\big) = 
\| \div \bff(v)\|_{0,\Omega_{k-1,k}}^2 +
 \alpha( \|v-u^{k-1}\|_{0, \Gamma_{k-1,k}}^2+\|v-g\|_{0, \Gamma^k_-}^2), 
\end{equation}
where $\alpha$ is a weight to be chosen empirically.

Unless otherwise stated, the integration mesh 
$\mathcal{T}_k$ is a uniform partition of $\Omega_{k-1,k}$ with $h=\delta=0.01$, and the discrete divergence operator defined in (\ref{dDO}) is based on the composite trapezoidal rule with $\hat{m}=\hat{n}=2$. Three-layer or four-layer neural network are employed for all test problems and are denoted by $d_in$-$n_1$-$n_2$(-$n_3$)-1 with $n_1$, $n_2$ and $n_3$ neurons in the respective first, second and third (for a four-layer NN)layers. The same network structure is used for all time blocks.

The network is trained by using the ADAM \cite{kingma2015} (a variant of the method of gradient descent) with either a fixed or an adaptive learning rate to iteratively solve the minimization problem in (\ref{discrete_minimization_functional-block}). Parameters of the first block is initialized by an approach introduced in \cite{LiuCai1}, and those for the current block is initialized by using the NN approximation of the previous block (see Remark 4.1 of \cite{Cai2021nonlinear}). 

The solution of the problem in (\ref{pde2}) and its corresponding NN approximation are denoted by $u^k$ and $u^k_{_\cT}$, respectively. Their traces are depicted on a plane of given time and exhibit capability of the numerical approximation in capturing shock/rarefaction. 

\subsection{Inviscid Burgers' equation}\label{sec5.1}

This section reports numerical results of the block space-time LSNN method for the one dimensional inviscid Burgers equation, where the spatial flux is ${\tilde\bff}(u) =f(u)= \frac12 u^2$.

\begin{table}[htbp]
\centering
\caption{Relative $L^2$ errors of Riemann problem {\em (}shock{\em )} for Burgers' equation}
\vspace{5pt}
\begin{tabular}{|l|l|l|}
\hline
Network structure &Block & $\frac{\|u^k-u^k_{_\cT}\|_0}{\|u^k\|_0}$ \\ \hline
 \multirow{3}{*}{2-10-10-1} & $\Omega_{0,1}$  &0.048774 \\ \cline{2-3}
 & $\Omega_{1,2}$  &0.046521 \\ \cline{2-3}
 & $\Omega_{2,3}$  &0.044616 \\ \hline
\end{tabular}
\centering
\label{riemann_shock_table}
\end{table}

\begin{figure}[htbp]
  \centering 
  \subfigure[Exact solution $u$ on $\Omega$]{ 
    \includegraphics[width=2in]{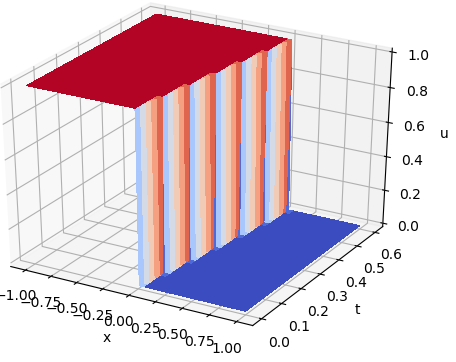}}
\hspace{0.5in}
    \subfigure[Traces at $t=0.2$ 
    ]{ 
    \includegraphics[width=2in]{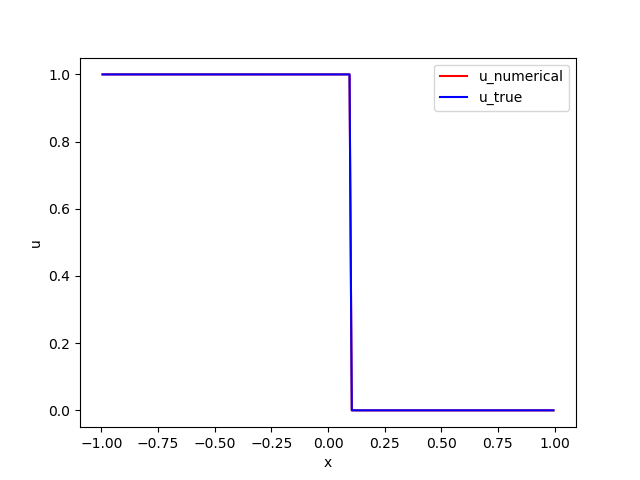}} 
  \\
\subfigure[Traces at $t=0.4$ 
]{ 
    \includegraphics[width=2in]{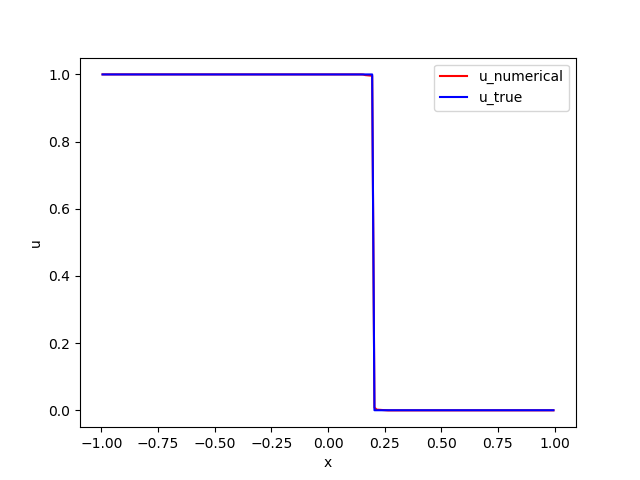}} 
     \hspace{0.5in} 
\subfigure[Traces at $t=0.6$ 
]{ 
    \includegraphics[width=2in]{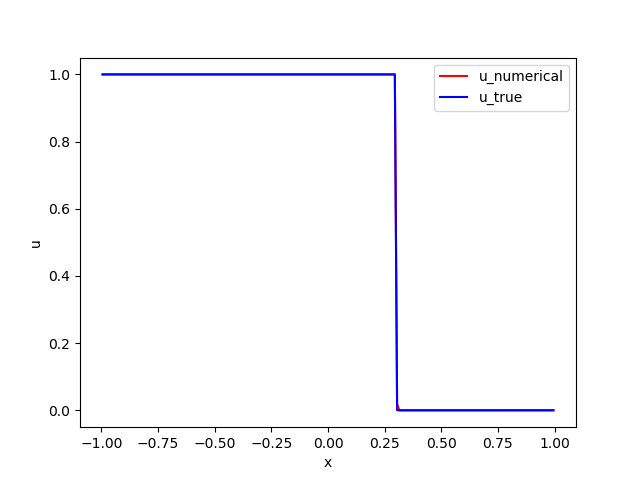}} 
  \caption{Approximation results of Riemann problem {\em (}shock{\em )} for Burgers' equation} 
  \label{riemann_shock_figure}
\end{figure}

The first two test problems are the Riemann problem with the initial condition: $u_0(x)=u(x,0)=u_{_L}$ if $x\leq 0$ or $u_{_R}$ if $x \ge 0$. 

\noindent{\bf Shock formation.} When $u_{_L}=1> 0=u_{_R}$, a shock is formed immediately with the shock speed $s=\left(u_{_L}+ u_{_R}\right)/2$. The first test problem is defined on a computational domain $\Omega =(-1,1) \times (0,0.6)$ with inflow boundary
\[
\Gamma_- = \Gamma_-^L\cup \Gamma_-^R\equiv \{(-1,t): t\in [0,0.6] \} \cup  \{(1,t): t\in [0,0.6]\}
\] 
and boundary conditions: $g=u_{_L}=1$ on $\Gamma_-^L$ and $g=u_{_R}=0$ on $\Gamma_-^R$. With $n_b=3$ blocks, weight $\alpha = 20$, a fixed learning rate $0.003$, and $30000$ iterations for each block, the relative errors in the $L^2$ norm are reported in Table \ref{riemann_shock_table}. Traces of the exact solution and numerical approximation on the planes $t=kT/n_b$ for $k=1,2, 3$ are depicted in Fig. \ref{riemann_shock_figure}(b)-(d), which clearly indicate that the LSNN method is capable of capturing the shock formation and its speed. Moreover, it approximates the solution well without oscillations. 

\begin{table}[htbp]
\centering
\caption{Relative $L^2$ errors of Riemann problem {\em (}rarefaction{\em )} for Burgers' equation}
\vspace{5pt}
\begin{tabular}{|l|l|l|}
\hline
Network structure &Block & $\frac{\|u^k-u^k_{_\cT}\|_0}{\|u^k\|_0}$ \\ \hline
\multirow{2}{*}{2-10-10-1} & $\Omega_{0,1}$&0.013387 \\ \cline{2-3}
& $\Omega_{1,2} $&0.010079 \\ \hline
\end{tabular}
\centering
\label{riemann_rare_table}
\end{table}

\begin{figure}[htbp]
  \centering
  \subfigure[Exact solution $u$ on $\Omega$]{      \includegraphics[width=1.65in]{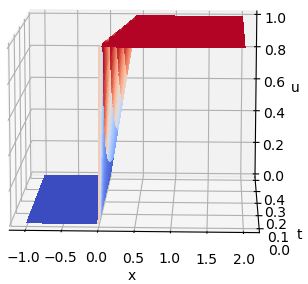}} 
    \subfigure[Traces at $t=0.2$ ]{ 
    \includegraphics[width=1.9in]{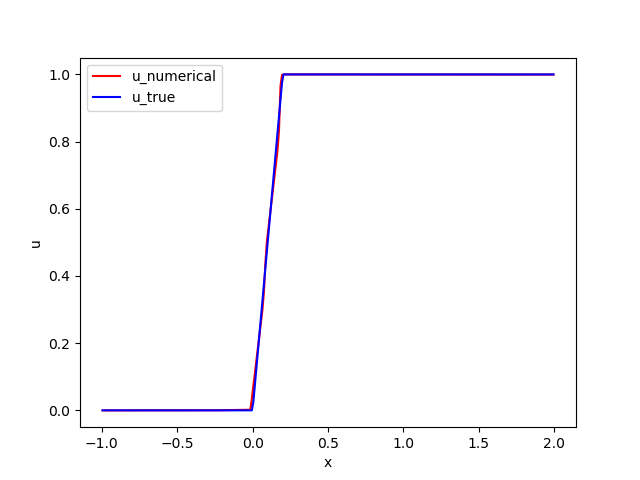}} 
  \hspace{0.06in} 
  \subfigure[Traces at $t=0.4$ ]{ 
    \includegraphics[width=1.9in]{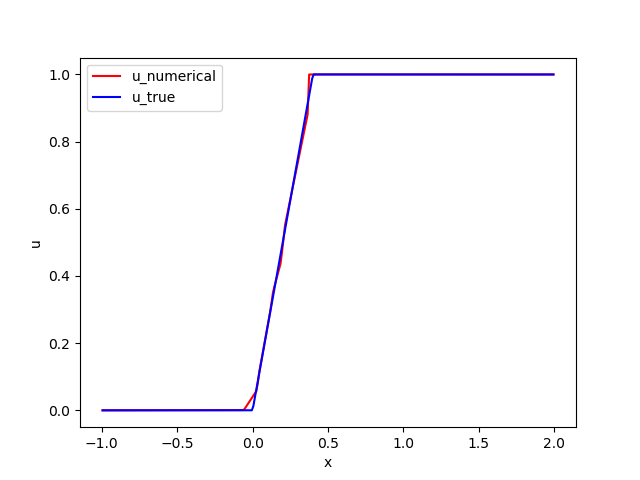}}
  \caption{Approximation results of Riemann problem {\em (}rarefaction{\em )} for Burgers' equation} 
  \label{riemann_rare_figure}
\end{figure}

\noindent{\bf Rarefaction waves}. When $u_{_L}=0 < 1=u_{_R}$, the range of influence of all points in $\R$ is a proper subset of $\R\times [0,\infty)$. Hence, the weak solution of the scalar hyperbolic conservation law is not unique. The second test problem is 
defined on a computational domain $\Omega =(-1,2) \times (0,0.4)$ with inflow boundary condition $g=0$ on $\Gamma_-=\{(-1,t): t\in [0,0.4] \}$.
As shown in Section~5.1.2 of \cite{Cai2021nonlinear}, the LSNN method using Roe’s scheme has a limitation to resolve the rarefaction. Numerical results of the LSNN method using the discrete divergence operator ($n_b=2$, $\alpha = 10$, a fixed learning rate $0.003$, and $40000$ iterations) are reported in Table \ref{riemann_rare_table}. 
Traces of the exact solution and numerical approximation on the planes $t=0.2$ and $t=0.4$ are depicted in Fig. \ref{riemann_rare_figure}. This test problem shows that the LSNN method using the \divt is able to compute the physically relevant vanishing viscosity solution (see, e.g., \cite{leveque1992numerical, thomas2013numerical}) without special treatment. This is possibly due to the fact that the LSNN approximation is continuous. 

\noindent {\bf Sinusoidal initial condition}. The third test problem has smooth initial condition $u_0(x) = 0.5 + \sin(\pi x)$ and is defined on the computational domain $\Omega = (0,2)\times (0,0.8)$ with inflow boundary 
\[
\Gamma_- = \Gamma_-^L\cup \Gamma_-^R\equiv \{(0,t): t\in [0,0.8] \} \cup  \{(2,t): t\in [0,0.8]\}.
\]
The shock of the problem appears at $t = 1/\pi \approx 0.318$. This is the same test problem as in Section~5.2 of \cite{Cai2021nonlinear} (see also \cite{leveque2020, Qiu2020}). The goal of this experiment is to compare numerical performances of the LSNN methods using the \divt introduced in this paper and the ENO scheme in \cite{Cai2021nonlinear}.


Since the solution of this problem is implicitly given, to accurately measure the quality of NN approximations, a benchmark reference solution $\hat{u}$ is generated using the traditional mesh-based method. In particular, the third-order accurate WENO scheme \cite{shu1998essentially} and the fourth-order Runge-Kutta method are employed for the respective spatial and temporal discretizations with a fine mesh ($\Delta x = 0.001$ and $\Delta t =0.0002$) on the computational domain $\Omega$. 

The LSNN using \divt is implemented with the same set of hyper parameters as in Section~5.2 of \cite{Cai2021nonlinear}, i.e., training weight $\alpha=5$ and an adaptive learning rate which starts with 0.005 and reduces by half for every 25000 iterations. Setting $n_b=16$ and on each time block, the total number of iterations is set as 50000 and the size of the NN model is 2-30-30-1. Although we observe some error accumulation when the block evolves for both the LSNN methods, the one using \divt performs better than that using ENO (see Table~\ref{sin_table} for the relative $L^2$ norm error and Fig.~\ref{sin_figure}(a)-(h) for graphs near the left side of the interface). 


\begin{table}[htbp]
\centering
\caption{Relative $L^2$ errors of Burgers' equation with a sinusoidal initial condition}
\vspace{5pt}
\begin{tabular}{|l |c | c | c |}
\hline
 Network structure &Block &
\multicolumn{1}{|p{4cm}|}{\centering LSNN using \divt \\ $\frac{\|u^k-u^k_{_\cT}\|_0}{\|u^k\|_0}$ } &
\multicolumn{1}{|p{4cm}|}{\centering LSNN using ENO \cite{Cai2021nonlinear} \\ $\frac{\|u^k-u^k_{_\cT}\|_0}{\|u^k\|_0}$ }
\\ 
\hline
\multirow{16}{*}{2-30-30-1} & $\Omega_{0,1}$ &0.010641 & 0.010461 \\ \cline{2-4}
 & $\Omega_{1,2}$ &0.011385 & 0.012517 \\ \cline{2-4}
 & $\Omega_{2,3}$ &0.012541 & 0.019772\\ \cline{2-4}
 & $\Omega_{3,4}$ &0.014351 & 0.022574\\ \cline{2-4}
 & $\Omega_{4,5}$ &0.016446 & 0.029011\\ \cline{2-4}
 & $\Omega_{5,6}$ &0.018634 & 0.038852\\ \cline{2-4}
 & $\Omega_{6,7}$ &0.031103 & 0.075888\\ \cline{2-4}
 & $\Omega_{7,8}$ &0.053114 & 0.078581\\ \cline{2-4}
 & $\Omega_{8,9}$ & 0.053562& --\\ \cline{2-4}
 & $\Omega_{9,10}$ &0.064933 &-- \\ \cline{2-4}
 & $\Omega_{10,11}$ &0.061354 & --\\ \cline{2-4}
 & $\Omega_{11,12}$ &0.077982 & --\\ \cline{2-4}
 & $\Omega_{12,13}$ &0.061145 & --\\ \cline{2-4}
 & $\Omega_{13,14}$ &0.070554 & --\\ \cline{2-4}
  & $\Omega_{14,15}$ &0.068539 & --\\ \cline{2-4}
 & $\Omega_{15,16}$ &0.065816 & --\\ \cline{2-4}
\hline
\end{tabular}
\centering
\label{sin_table}
\end{table}


\begin{figure}[htbp]
  \centering 
    \subfigure[Traces at $t=0.05$ 
    ]{ 
    \includegraphics[width=1.7in]{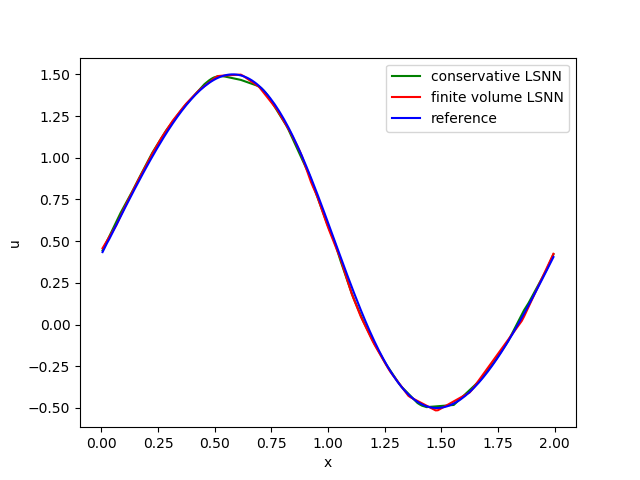}} 
  \hspace{0.2in} 
  \subfigure[Traces at $t=0.1$ 
  ]{ 
    \includegraphics[width=1.7in]{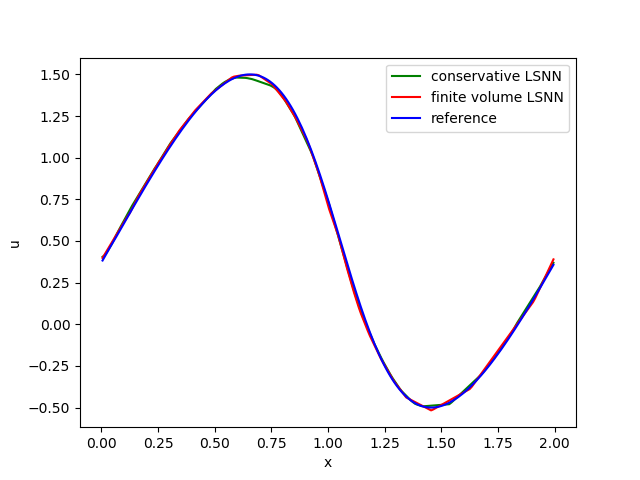}} 
  \hspace{0.2in} 
  \subfigure[Traces at $t=0.15$ 
  ]{ 
    \includegraphics[width=1.7in]{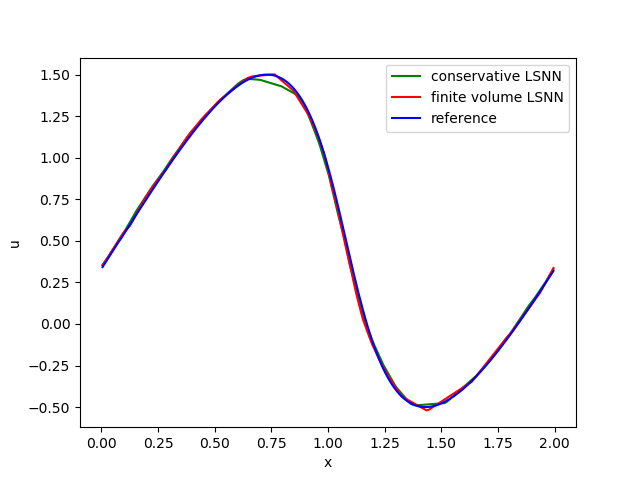}}
\\
      \subfigure[Traces at $t=0.2$ 
      ]{ 
    \includegraphics[width=1.7in]{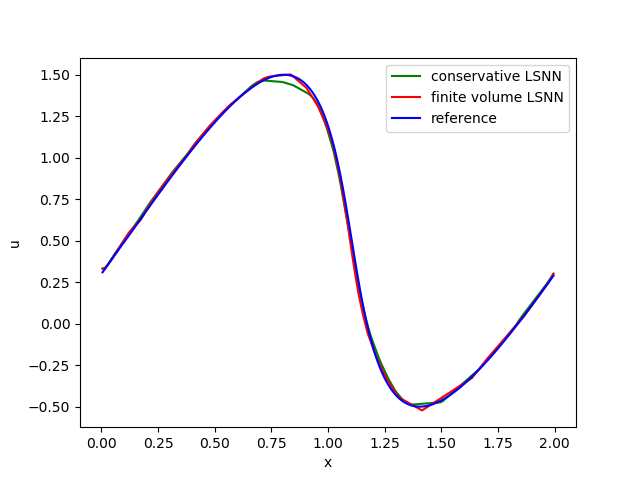}}
     \hspace{0.2in} 
\subfigure[Traces at $t=0.25$ 
]{ 
    \includegraphics[width=1.7in]{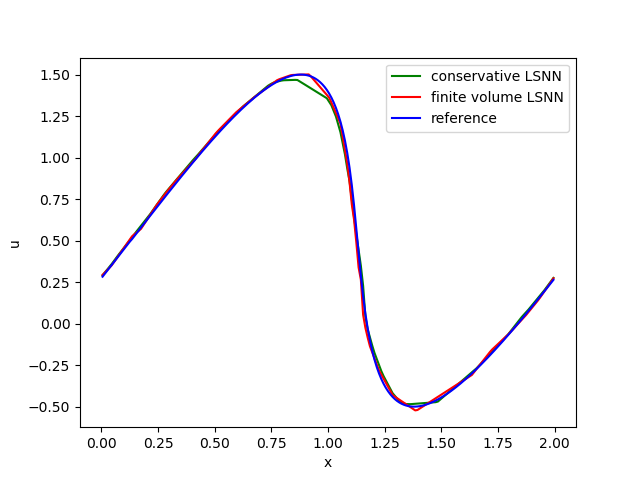}} 
    \hspace{0.2in} 
\subfigure[Traces at $t=0.3$ 
]{ 
    \includegraphics[width=1.7in]{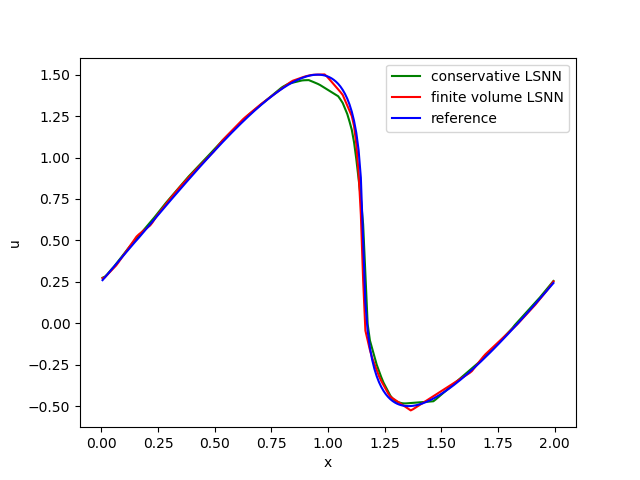}}
\\
     \subfigure[Traces at $t=0.35$ 
     ]{ 
    \includegraphics[width=1.7in]{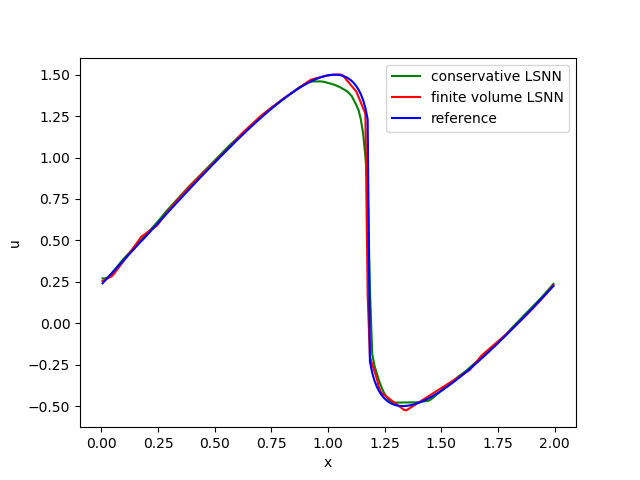}}
 \hspace{0.2in} 
\subfigure[Traces at $t=0.4$ 
]{ 
    \includegraphics[width=1.7in]{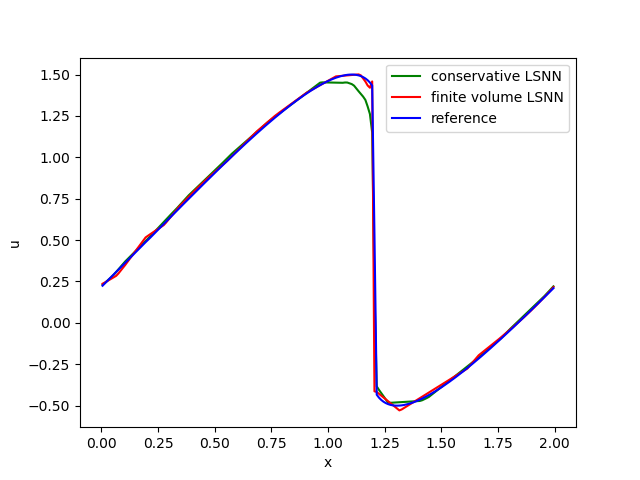}}
     \hspace{0.2in} 
\subfigure[Traces at $t=0.8$ 
]{ 
    \includegraphics[width=1.7in]{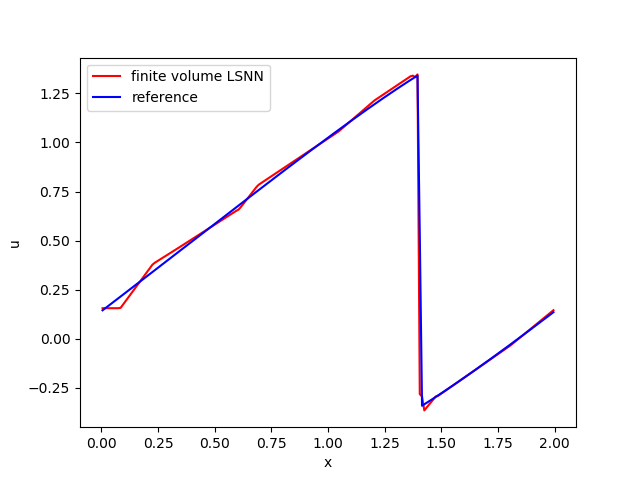}}
  \caption{Approximation results of Burgers' equation with a sinusoidal initial condition} 
  \label{sin_figure}
\end{figure}

\begin{table}[htbp]
\centering
\caption{Relative $L^2$ errors of the problem with $f(u)=\frac{1}{4}u^4$ using the composite trapezoidal rule {\em (\ref{integration})}}
\vspace{5pt}
\begin{tabular}{|c|c|c|c|}
\hline
\multirow{2}{*}{Time block} & \multicolumn{3}{c|}{Number of sub-intervals} \\ \cline{2-4} 
                            & $\hat{m}=\hat{n}=2$  & $\hat{m}=\hat{n}=4$  & $\hat{m}=\hat{n}=6$  \\ \bottomrule
\multicolumn{1}{|l|}{ $\Omega_{0,1}$} &0.067712 &0.010446 &0.004543  \\ \hline
\multicolumn{1}{|l|}{ $\Omega_{1,2}$} &0.108611 &0.008275 &0.009613  \\ \hline
\end{tabular}
\label{trapi_convergence}
\end{table}

\begin{table}[htbp]
\centering
\caption{Relative $L^2$ errors of the problem with $f(u)=\frac{1}{4}u^4$ using the composite mid-point rule {\em (\ref{integration})}}
\vspace{5pt}
\begin{tabular}{|c|c|c|c|}
\hline
\multirow{2}{*}{Time block} & \multicolumn{3}{c|}{Number of sub-intervals} \\ \cline{2-4} 
                            & $\hat{m}=\hat{n}=2$  & $\hat{m}=\hat{n}=4$  & $\hat{m}=\hat{n}=6$  \\ \bottomrule
\multicolumn{1}{|l|}{ $\Omega_{0,1}$} &0.096238 &0.007917 &0.003381  \\ \hline
\multicolumn{1}{|l|}{ $\Omega_{1,2}$} &0.159651 &0.007169 &0.005028  \\ \hline
\end{tabular}
\label{midpoint_convergence}
\end{table}

\begin{figure}[ht]
    \subfigure[Traces at $t=0.2$ 
    (trapezoidal) 
    ]{ 
    \includegraphics[width=1.8in]{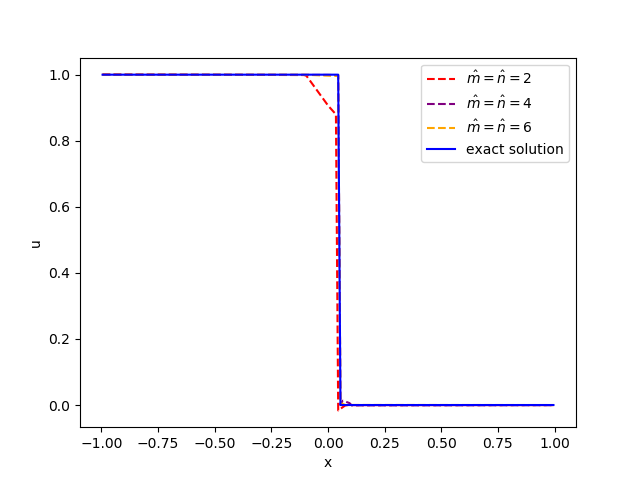}}
    \hspace{0.1in}
  \subfigure[Zoom-in plot near the discontinuous interface of sub-figure (a)]{
    \includegraphics[width=1.8in]{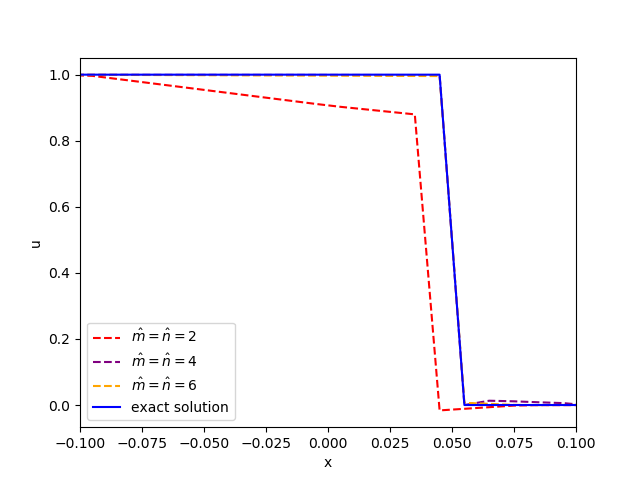}} 
     \hspace{0.1in}
  \subfigure[Traces at $t=0.4$ (trapezoidal) 
  ]{
    \includegraphics[width=1.8in]{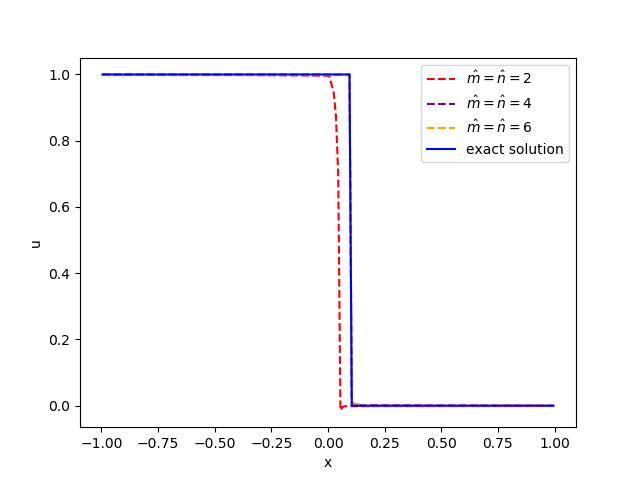}}
    \\
    \subfigure[Traces at $t=0.2$ (mid-point) ]{ 
    \includegraphics[width=1.8in]{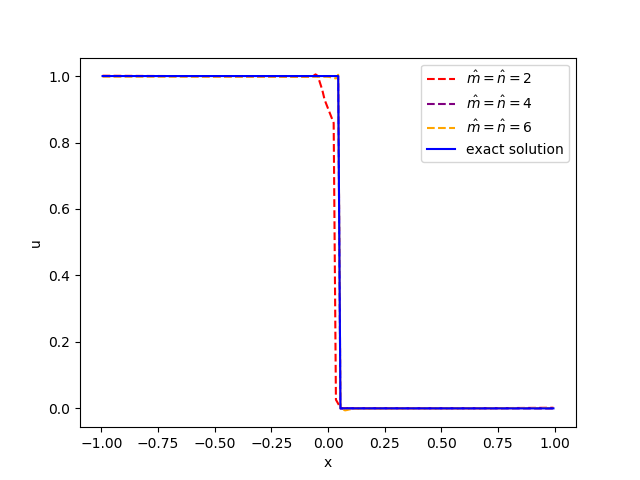}}
     \hspace{0.1in}
  \subfigure[Zoom-in plot near the discontinuous interface of sub-figure (d)]{
    \includegraphics[width=1.8in]{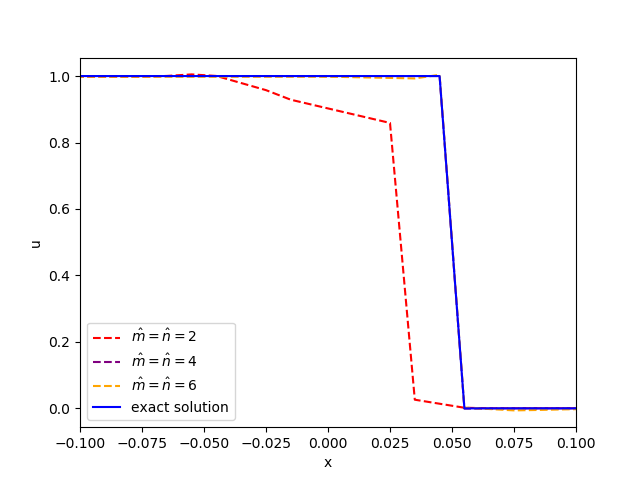}}
     \hspace{0.1in}
  \subfigure[Traces at $t=0.4$ (mid-point)]{
    \includegraphics[width=1.8in]{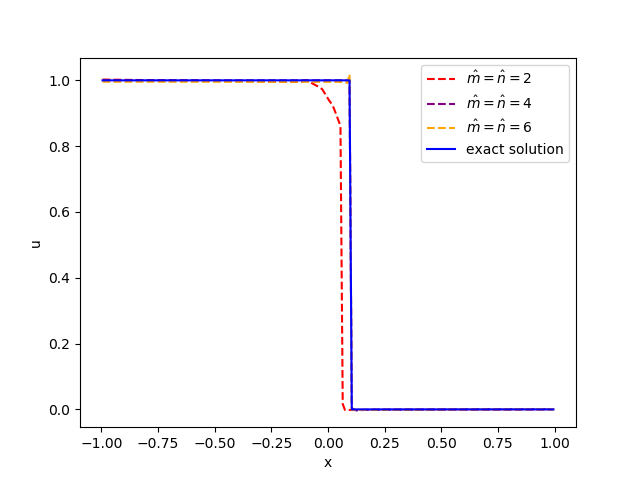}}
  \caption{Numerical results of the problem with $f(u)=\frac{1}{4}u^4$ using the composite trapezoidal and mid-point rules} 
  \label{convex_flux_fig}
\end{figure}

\subsection[Riemann problem with a convex fhttps://www.overleaf.com/project/601b08e4bc00a310266cb74clux]{Riemann problem with $f(u)=\frac14 u^4$}\label{sec5.2}
The goals of this set of numerical experiments are twofold. First, we compare the performance of the LSNN method using 
the composite trapezoidal/mid-point rule in (\ref{integration}). Second, we investigate the impact of the number of sub-intervals of the composite quadrature rule on the accuracy of the LSNN method.

The test problem is the Riemann problem with a convex flux ${\bff}(u) = (f(u), u)=(\frac14 u^4, u)$ and the initial condition $u_{_L}=1 >0=u_{_R}$. 
The computational domain is chosen to be $\Omega = (-1,1)\times (0,0.4)$. 
Relative $L^2$ errors of the LSNN method using the \divt (2-10-10-1 NN model, $n_b=2$, $\alpha = 20$, a fixed learning rate $0.003$ for the first $30000$ iterations and $0.001$ for the remaining) are reported in Tables~\ref{trapi_convergence} and \ref{midpoint_convergence}; and traces of the exact and numerical solutions are depicted in Fig.~\ref{convex_flux_fig}.


Clearly, Tables~\ref{trapi_convergence} and \ref{midpoint_convergence} indicate that the accuracy of the LSNN method depends on the number of sub-intervals ($\hat{m}$ and $\hat{n}$) for the composite quadrature rule; i.e., the larger $\hat{m}$ and $\hat{n}$ are, the more accurate the LSNN method is. Moreover, the accuracy using the composite trapezoidal and mid-point rules in the LSNN method is comparable. 



\subsection[Riemann problem with a non-convex flux]{Riemann problem with non-convex fluxes}\label{sec5.3}



\begin{table}[htbp]
\centering
\caption{Relative $L^2$ errors of Riemann problem with a non-convex flux $f(u)=\frac13 u^3$}
\vspace{5pt}
\begin{tabular}{|l|l|c|c|}
\hline
Network structure &Block & $\frac{\|u^k-u^k_{_\cT}\|_0}{\|u^k\|_0}$ \\ \hline
\multirow{4}{*}{2-64-64-64-1} & $\Omega_{0,1}$ &   0.03277\\ \cline{2-3}
 & $\Omega_{1,2}$ & 0.03370 \\ \cline{2-3}
 & $\Omega_{2,3}$ & 0.03450 \\ \cline{2-3}
 & $\Omega_{3,4}$ & 0.03578  \\ \hline
\end{tabular}
\centering
\label{nonconvex_table}
\end{table}


\begin{figure}[htbp]
\centering
    \subfigure[Traces at $t=0.1$]{ 
    \includegraphics[width=1.8in]{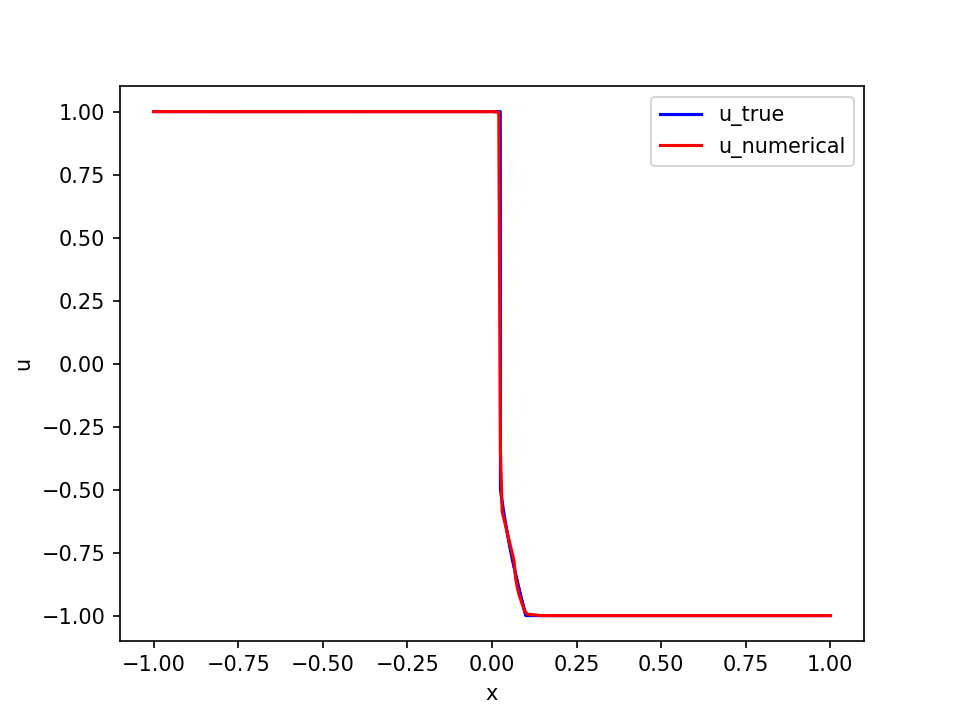}}
    \hspace{0.05in}
  \subfigure[Traces at $t=0.2$]{
    \includegraphics[width=1.8in]{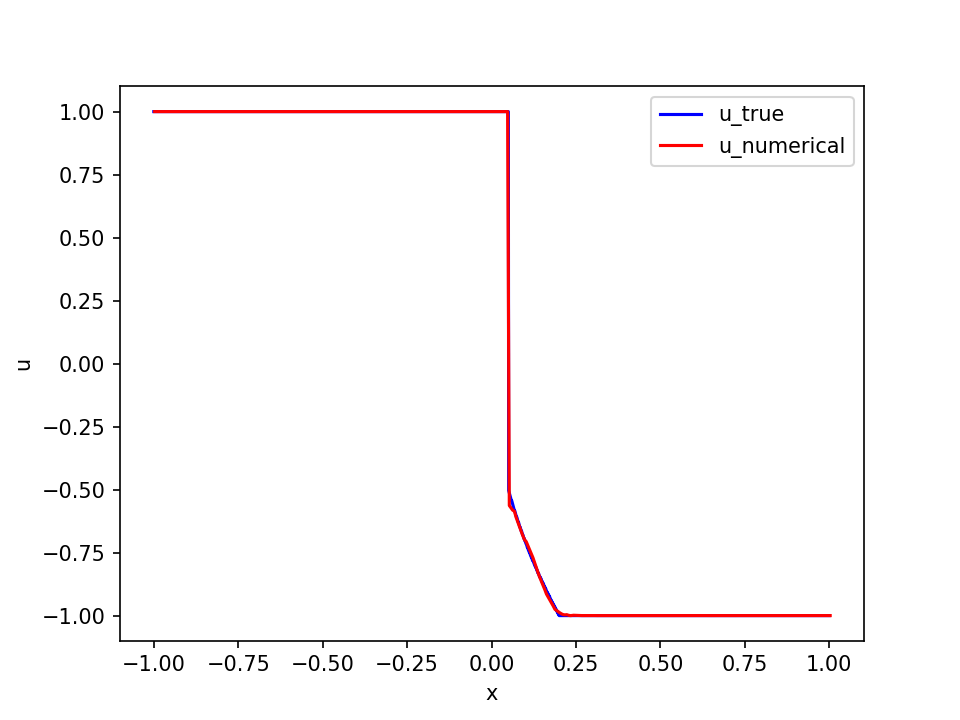}}
    \hspace{0.05in}
  \subfigure[Traces at $t=0.3$ ]{
    \includegraphics[width=1.8in]{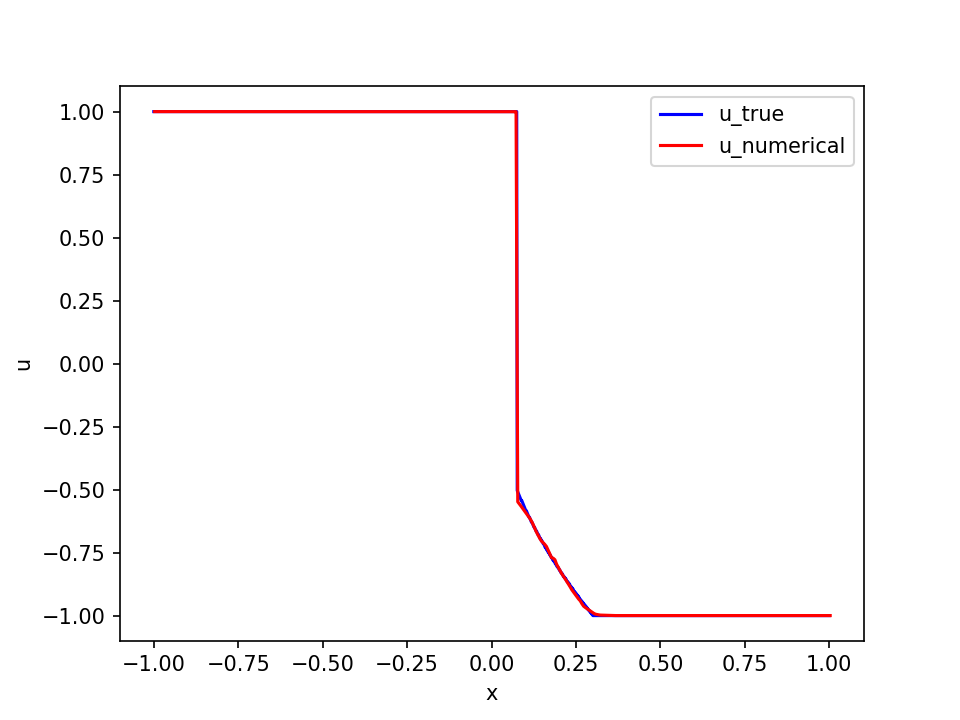}}\\
    \subfigure[Traces at $t=0.4$]{ 
    \includegraphics[width=1.8in]{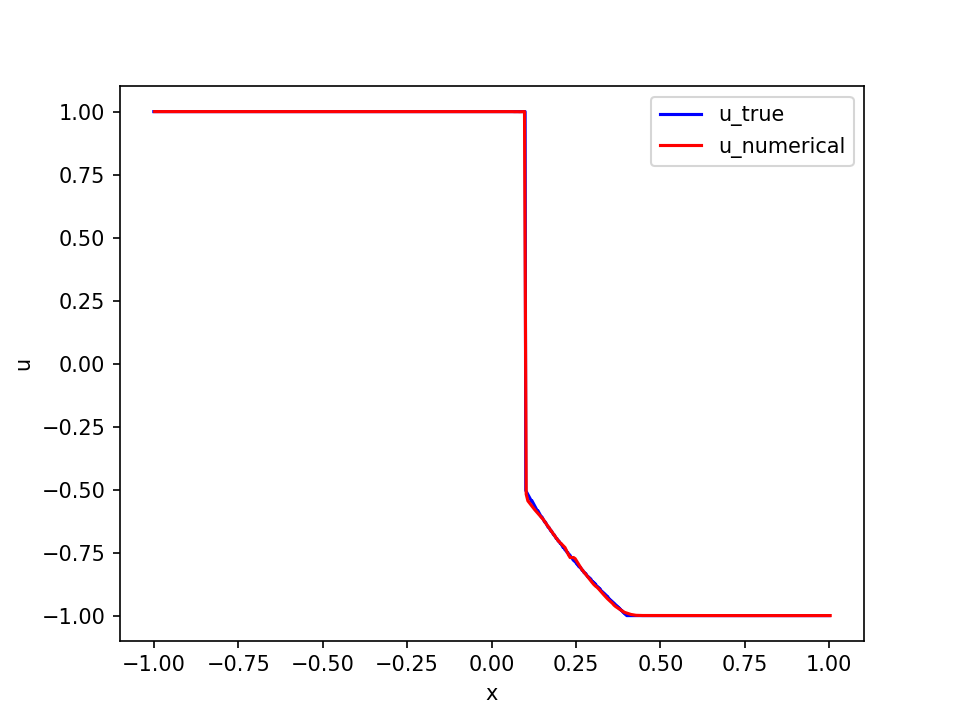}}
\hspace{0.05in}
  \subfigure[Numerical Solution $u_N$ on $\Omega$]{
    \includegraphics[width=2in]{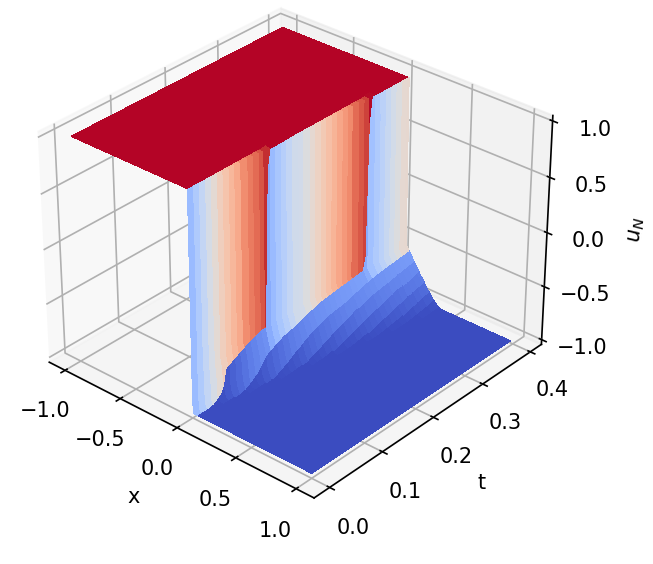}}
  \caption{Numerical results of Riemann problem with a non-convex flux $f(u)=\frac13 u^3$} 
  \label{nonconvex_figure}
\end{figure}


The test problem for a non-convex flux is a modification of the test problem in Section~\ref{sec5.2} by replacing the flux with $f(u)=\frac{1}{3}u^3$ and the initial condition with $u_L=1 >-1=u_R$.
The Riemann solution consists partly of a rarefaction wave together with a shock wave which brings a new level of challenge with a compound wave. The exact solution is obtained through Osher's formulation \cite{osher1984} which has a shock speed s=0.25 and a shock jump from $1$ to $-0.5$ when $t>0$.


The block space-time LSNN method using the \divt with $\hat{m}=\hat{n}=4$ is utilized for this problem. Four time blocks are computed on the temporal domain (0, 0.4) and a relative larger network structure (2-64-64-64-1) is tested with a smaller integration mesh size $h=\delta=0.005$ to compute the compound wave more precisely. We tune the hyper parameter $\alpha=200$, and all time blocks are computed with a total of 60000 iterations (learning rate starts with 1e-3 and decay to $20\%$ every 20000 iterations). Due to the random initial guess for the second hidden layer parameters, the experiment is replicated several times. Similar results are obtained as the best result reported in Table \ref{nonconvex_table} and Fig. \ref{nonconvex_figure} (a)-(e). These experiments demonstrate that the LSNN method can capture the compound wave for non-convex flux problems as well. 



\subsection{Two-dimensional problem}\label{sec5.4}
Consider a two-dimensional inviscid Burgers equation,  where the spatial flux vector field is
$\tilde{\bff}(u)=\frac12 (u^2, u^2)$. Given a piece-wise constant initial data 
\begin{equation}\label{2dburgers}
    u_0(x,y)=\left\{\begin{array}{r l}
    -0.2, & \mbox{if } \; x < 0.5\; \mbox{ and }\;y >0.5,\\
    -1.0, & \mbox{if }\; x > 0.5\; \mbox{ and }\; y >0.5,\\
    0.5, & \mbox{if } \; x < 0.5\; \mbox{ and } \; y < 0.5,\\
    0.8, & \mbox{if } \; x > 0.5\; \mbox{ and }\; y < 0.5,
    \end{array}\right.
\end{equation}
this problem has an exact solution given in \cite{GUERMOND2014}.

The test problem is set on computational domain $\Omega = (0,1)^2\times (0, 0.5)$ with inflow boundary conditions prescribed by using the exact solution. Our numerical result using a 4-layer LSNN (3-48-48-48-1) with 3D \divt ($\hat{m}=\hat{n}=\hat{k}=2$) are reported in Table \ref{riemann_2d_table}. The corresponding hyper parameters setting is as follows: $n_b=5$, $\alpha = 20$, the first time block is trained with $30000$ iteration where the first $10000$ iterations are using learning rate $0.003$ and the rest iterations are trained using learning rate of $0.001$; all remaining time blocks are trained with $20000$ iterations using fixed learning rate of $0.001$. Fig.\,\ref{fig_burger2d} presents the graphical results at time $t= 0.1$, $0.3$, and $0.5$. This experiment shows that the proposed LSNN method can be extended to two dimensional problems and can capture the shock and rarefaction waves in two dimensions.

\begin{table}[htbp]
\centering
\caption{Relative $L^2$ errors of Riemann problem {\em (}shock{\em )} for 2D Burgers' equation}
\vspace{5pt}
\begin{tabular}{|l|l|l|}
\hline
Network structure &Block & $\frac{\|u^k-u^k_{_\cT}\|_0}{\|u^k\|_0}$ \\ \hline
 \multirow{3}{*}{3-48-48-48-1} & $\Omega_{0,1}$  & 0.093679 \\ \cline{2-3}
 & $\Omega_{1,2}$  & 0.121375 \\ \cline{2-3}
 & $\Omega_{2,3}$  & 0.163755\\ \cline{2-3}
 & $\Omega_{3,4}$  & 0.190460\\ \cline{2-3}
 & $\Omega_{4,5}$  & 0.213013\\ \hline
 
\end{tabular}
\centering
\label{riemann_2d_table}
\end{table}

\begin{figure}[htbp]
\centering
    \subfigure[$t=0.1$]{ 
    \includegraphics[width=2in]{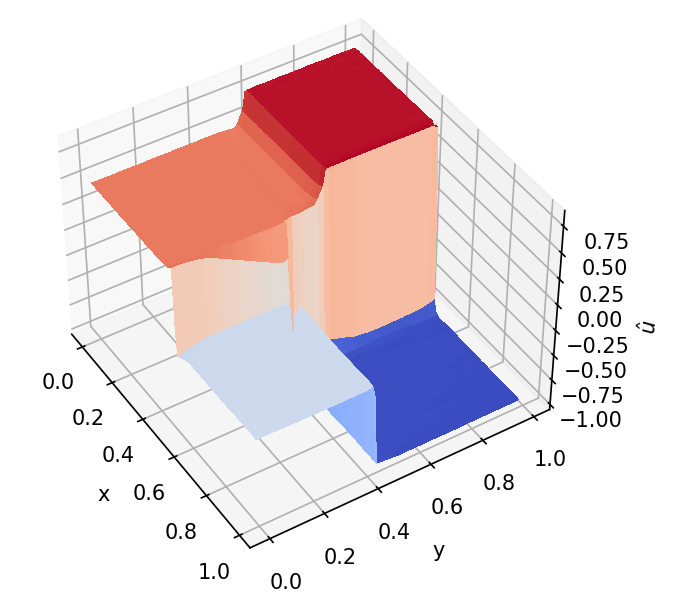}}
  \subfigure[$t=0.3$ ]{
    \includegraphics[width=1.8in]{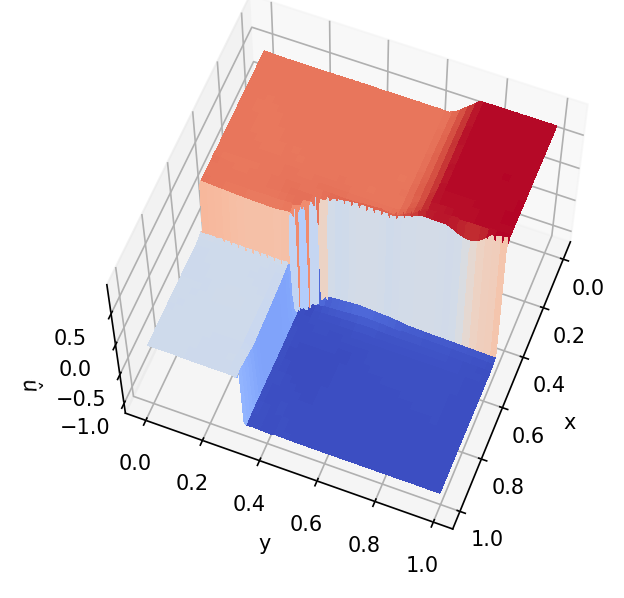}}
    \subfigure[$t=0.5$]{ 
    \includegraphics[width=1.8in]{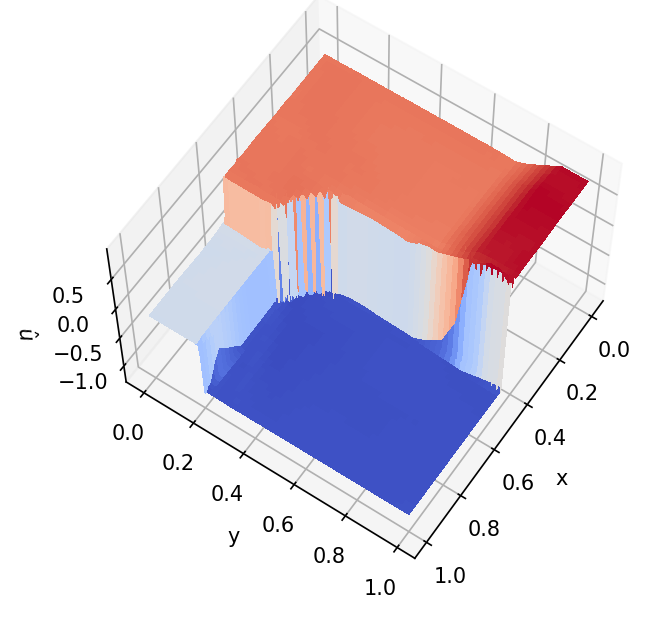}}
  \caption{Numerical results of $2D$ Burgers' equation.} 
  \label{fig_burger2d}
\end{figure}

\section{Discussion and Conclusion}\label{sec6}

The ReLU neural network provides a new class of approximating functions that is ideal for approximating discontinuous functions with unknown interface location \cite{Cai2021linear}. Making use of this unique feature of neural networks, this paper studied the least-squares ReLU neural network (LSNN) method for solving scalar nonlinear hyperbolic conservation laws. 

In the design of the LSNN method for HCLs, the numerical approximation of differential operators is a critical factor, and standard numerical or automatic differentiation along coordinate directions can often lead to a failed NN-based method. To overcome this challenge, this paper introduced a new discrete divergence operator \divt based on its physical meaning. 


Numerical results for several test problems show that the LSNN method using the \divt does overcome limitations of the LSNN method with conservative flux in \cite{Cai2021nonlinear}. Moreover, for the one dimensional test problems with fluxes $f(u)=\frac14 u^4$ and $\frac13 u^3$, the accuracy of the method may be improved greatly by using enough number of sub-intervals in the composite trapezoidal/mid-point quadrature. 

Compared to other NN-based methods like the PINN and its variants, the LSNN method introduced in this paper free of any penalization such as the entropy, total variation, and/or artificial viscosity, etc. Usually, choosing proper penalization constants can be challenging in practice and it affects the accuracy, efficiency, and stability of the method.

Even though the number of degrees of freedom for the LSNN method is several order of magnitude less than those of traditional mesh-based numerical methods, training NN is computationally intensive and complicated. For a network with more than one hidden layer, random initialization of the parameters in layers beyond the first hidden layer would cause some uncertainty in training NN (iteratively solving the resulting non-convex optimization) as observed in Section~\ref{sec5.2}. This issue plus designation of a proper architecture of NN would be addressed in a forthcoming paper using the adaptive network enhancement (ANE) method developed in \cite{LiuCai1,LiuCai2,Cai2021DeepAdaptive}.

\bibliographystyle{ieee}
\bibliography{main.bbl}

\bigskip
\bigskip
\bigskip
\bigskip
\bigskip

\section{Appendix}\label{secA}

In the appendix, we provide the proofs of Lemmas ~\ref{3.2} and \ref{3.3}. First, denote the integral and the mid-point/trapezoidal rule of a function $\varphi$ over an interval $[0,\rho]$ by
\[
I(\varphi)= \int^{\rho}_{0}\varphi(s)\,ds\quad\mbox{and}\quad 
Q(\varphi;0,\rho,1)=\left\{\begin{array}{ll}
    \rho\,\varphi(\rho/2), & \mbox{midpoint},\\[2mm]
    \dfrac{\rho}{2}\big(\varphi(0)+\varphi(\rho)\big), & \mbox{trapezoidal},
    \end{array}\right. 
\]
respectively. Let $p,\,q\in (1,\infty]$ such that $1/p+1/q=1$. It is easy to show the following error bounds:
 \begin{equation}\label{Err_Smooth}
 \big|I(\varphi)-Q(\varphi;0,\rho,1)\big|\leq \left\{\begin{array}{ll}
 C\rho^{2+ 1/q} \|\varphi^{\prime\prime}\|_{L^p(0,\rho)}, & \mbox{if } \varphi\in C^2(0,\rho),\\[2mm]
 C\rho^{1+ 1/q} \|\varphi^{\prime}\|_{L^p(0,\rho)}, & \mbox{if } \varphi\in C^1(0,\rho).
 \end{array}\right.
 \end{equation}

\medskip

\noindent \begin{myproof}{{\em Lemma}}{{\em \ref{3.2}}} We prove Lemma~\ref{3.2} only for the mid-point rule because it may be proved in a similar fashion for the trapezoidal rule. To this end, denote uniform partitions of the intervals $[x_{i},x_{i+1}]$ and $[t_{j},t_{j+1}]$ by 
\[
x_{i}=x^0_i<x_i^1<\cdots<x^{\hat{m}}_i=x_{i+1},
\mbox{and}\quad  
t_{j}=t^0_j<t_j^1<\cdots <t^{\hat{n}}_j=t_{j+1},
\]
respectively, where $x^k_i=x_i+k\hat{h}$ and $t^k_j=t_j+k\hat{\delta}$; and $\hat{h}=h/\hat{m}$ and $\hat{\delta}=\delta/\hat{n}$ are the numerical integration mesh sizes. By (\ref{Err_Smooth}), we have
\begin{eqnarray*}
    \left|\int^{t^{k+1}_j}_{t^k_j} \sigma(x_{i},x_{i+1};t)\,dt
 - \hat{\delta}\sigma(x_{i},x_{i+1};t^{k+1/2}_j)\right|
 &\leq & C\,\hat{\delta}^{2+1/q}\|\sigma_{tt}(x_{i},x_{i+1};\cdot)\|_{L^p(t^k_j,t^{k+1}_j)},\\[2mm]
 \mbox{and}\quad \left| \int^{x^{k+1}_i}_{x^k_i}u(x;t_{j},t_{j+1})\,dx - \hat{h}u(x^{k+1/2}_{i};t_{j},t_{j+1})\right|
 &\leq & C\,\hat{h}^{2+1/q}\|u_{xx}(\cdot;t_{j},t_{j+1})\|_{L^p(x^k_i,x^{k+1}_i)},
\end{eqnarray*}
which, together with (\ref{SI}), (\ref{dDO}), and the triangle and the H\"{o}lder inequalities, implies 
\begin{eqnarray*}
&& |K_{ij}|^{1/q}\left\| \div{\!\!_{_\cT}}  \bff(u)-{{\text{avg}}}_{_\cT}\div \bff (u) \right\|_{L^p(K_{ij})}=|K_{ij}|\Big| {\text{avg}}_{K_{ij}}\div \bff (u) -\div{\!\!_{_\cT}} \bff\big(u(\bm_{ij})\Big| \\[2mm]
&\leq& C \left\{h\hat{\delta}^{2+1/q} \sum^{\hat{n}-1}_{k=0}
\|\sigma_{tt}(x_{i},x_{i+1};\cdot)\|_{L^p(t^k_j,t_j^{k+1})}  +\delta\hat{h}^{2+1/q}\sum^{\hat{m}-1}_{k=0} \|u_{xx}(\cdot;t_{j},t_{j+1})\|_{L^p(x^k_i,x_i^{k+1})}\right\}\\[2mm]
&\leq& C \left\{h\hat{\delta}^{2+1/q} \hat{n}^{1/q}
\|\sigma_{tt}(x_{i},x_{i+1};\cdot)\|_{L^p(t_j,t_{j+1})}  + \delta\hat{h}^{2+1/q}\hat{m}^{1/q}\|u_{xx}(\cdot;t_{j},t_{j+1})\|_{L^p(x_i,x_{i+1})}\right\}.
\end{eqnarray*}
This completes the proof of Lemma~\ref{3.2}.
\end{myproof}

\smallskip


To prove Lemma~\ref{3.3}, we need to estimate an error bound of numerical integration for piece-wise smooth and discontinuous integrant over interval $[0,\rho]$.


 
\begin{lemma}\label{7.1} For any $0<\hat{\rho} <\rho/2$, assume that $\varphi\in C^1\big((0,\hat{\rho})\big)\cap C^1\big((\hat{\rho},\rho)\big)$ is a piece-wise $C^1$ function. Denote by $j_\varphi=|\varphi(\hat{\rho}^+)-\varphi(\hat{\rho}^-)|$ the jump of $\varphi(s)$ at $s=\hat{\rho}$. 
Then there exists a positive constant $C$ such that
\begin{eqnarray}\nonumber
\big|I(\varphi)-Q(\varphi;0,\rho,1)\big|
&\leq &
C \rho^{1+1/q}\|\varphi^{\prime}\|_{L^p\big((0,\rho)\setminus \{\hat{\rho}\}\big)} + \left\{\begin{array}{ll}
 \hat{\rho}\,j_\varphi , & \mbox{mid-point},\\[2mm]
 \left|\dfrac{\rho}{2}-\hat{\rho}\right|\,j_\varphi, & \mbox{trapezoidal}
 \end{array}\right. \\[2mm] \label{Err_discont} 
 &\leq &
C \rho^{1+1/q}\|\varphi^{\prime}\|_{L^p\big((0,\rho)\setminus \{\hat{\rho}\}\big)} + \dfrac{\rho}{2} j_\varphi.
\end{eqnarray}
\end{lemma} 
 
\begin{proof}
Denote the linear interpolant of $\varphi$ on the interval $[0,\rho]$ by $\varphi_1(s)=\varphi(0)\,\dfrac{\rho-s}{\rho} + \varphi(\rho)\, \dfrac{s}{\rho}$.
For any $s\in (0,\hat{\rho})$, by the fact that $\varphi(0)-\varphi_1(0)=0$, a standard argument on the error bound of interpolant yields that there exists a $\xi_-\in (0,\hat{\rho})$ such that 
\[ 
\varphi(s)-\varphi_1(s)= \varphi^\prime(\xi_-) s -\dfrac{s}{\rho}(\varphi(\rho)-\varphi(0)), 
\] 
which implies \[ \int^{\hat{\rho}}_{0}(\varphi(s)-\varphi_1(s))\,ds = \int^{\hat{\rho}}_{0}\varphi^\prime(\xi_-)s\,ds - \dfrac{\hat{\rho}^2}{2\rho}\left(\varphi(\rho)-\varphi(0)\right). \] In a similar fashion, there exists a $\xi_-\in (\hat{\rho},\rho)$ such that \[ \int^{\rho}_{\hat{\rho}}(\varphi(s)-\varphi_1(s))\,ds = \int^{\rho}_{\hat{\rho}}\varphi^\prime(\xi_+)(s-\rho)\,ds + \dfrac{(\rho-\hat{\rho})^2}{2\rho}\left(\varphi(\rho)-\varphi(0)\right). \] 
Combining the above inequalities and using the triangle and the H\"{o}lder inequalities give 
\begin{eqnarray*} \nonumber
\big|I(\varphi)-Q_t(\varphi)\big|&=&\left| \int^{\hat{\rho}}_{0}\!\!\varphi^\prime(\xi_-)sds + \int^{\rho}_{\hat{\rho}}\!\!\varphi^\prime(\xi_+)(s-\rho)ds + \dfrac{\rho-2\hat{\rho}}{2}\left(\varphi(\rho)-\varphi(0)\right)\right|\\[2mm]  \label{3.5a}
&\leq & 
\dfrac{1}{(1+q)^{1/q}}\rho^{1+1/q}\left(\|\varphi^{\prime}\|_{L^p(0,\hat{\rho})}+ \|\varphi^{\prime}\|_{L^p(\hat{\rho},\rho)}\right) +\left|\dfrac{\rho}{2}-\hat{\rho}\right|\,\left|\varphi(\rho)-\varphi(0)\right| \\[2mm]
 &\leq & \dfrac{2^{1/q}}{(1+q)^{1/q}}\rho^{1+1/q}\|\varphi^{\prime}\|_{L^p\big((0,\rho)\setminus \{\hat{\rho}\}\big)} +\left|\dfrac{\rho}{2}-\hat{\rho}\right|\,\left|\varphi(\rho)-\varphi(0)\right|.
\end{eqnarray*} 
It follows from the triangle and the H\"{o}lder inequalities that 
\begin{eqnarray*} 
 \left|\varphi(\rho)-\varphi(0)\right| &\leq & \left|\int_{\hat{\rho}}^\rho\varphi^\prime(s)\,ds\right| + \left|\int_0^{\hat{\rho}}\varphi^\prime(s)\,ds\right| +j_\varphi \\[2mm] 
&\leq &\rho^{1/q}\left(\|\varphi^{\prime}\|_{L^p(0,\hat{\rho})}+ \|\varphi^{\prime}\|_{L^p(\hat{\rho},\rho)}\right) + j_\varphi
\leq \left(2\rho\right)^{1/q} \|\varphi^{\prime}\|_{L^p((0,\rho)\setminus \{\hat{\rho}\})} + j_\varphi.
 \end{eqnarray*} 
Now, the above two inequalities and the fact that $\left|\dfrac{\rho}{2}-\hat{\rho}\right|\leq \dfrac{\rho}{2}$ imply (\ref{Err_discont}) for the trapezoidal rule. 

To prove the validity of (\ref{Err_discont}) for the mid-point rule, note that for any $s\in (0,\hat{\rho})$ we have
\begin{eqnarray*}
\varphi(s)-\varphi(\rho/2)&=& \int^s_{\hat{\rho}}\varphi^\prime(s)\,ds +\int_{\rho/2}^{\hat{\rho}}\varphi^\prime(s)\,ds +  \varphi(\hat{\rho}^-)-\varphi(\hat{\rho}^+)\\[2mm]
&\leq & (\hat{\rho}-s)^{1/q}\|\varphi^{\prime}\|_{L^p(s,\hat{\rho})} + (\rho/2-\hat{\rho})^{1/q}\|\varphi^{\prime}\|_{L^p(\hat{\rho},\rho/2)} +  \varphi(\hat{\rho}^-)-\varphi(\hat{\rho}^+),  
\end{eqnarray*}
which, together with the triangle inequality, implies
\[
\left|\int^{\hat{\rho}}_0\big(\varphi(s)-\varphi(\rho/2)\big)\,ds\right| \leq \left(\dfrac{\rho}{2}\right)^{1+1/q} \left(\|\varphi^\prime\|_{L^p(0,\hat{\rho})} + \|\varphi^\prime\|_{L^p(\hat{\rho},\rho/2)}\right) +\hat{\rho} j_\varphi .
\]
Similarly, we have
\[
\left|\int_{\hat{\rho}}^\rho\big(\varphi(s)-\varphi(\rho/2)\big)\,ds\right| \leq \dfrac{2q}{1+q}\left(\dfrac{\rho}{2}\right)^{1+1/q} \|\varphi^\prime\|_{L^p(\hat{\rho},\rho)}.
\]
Now, (\ref{Err_discont}) for the mid-point rule follows from the triangle inequality and the above two inequalities. This completes the proof of the lemma.
\end{proof}

Now, we are ready to prove the validity of Lemma~\ref{3.3}.

\smallskip
 
\begin{myproof}{\em Lemma}{\em \ref{3.3}} 
By the assumption, the discontinuous interface $\Gamma_{ij}$ intercepts two horizontal edges at $(\hat{x}_i^l,t_l)$ for $l=j,j+1$. Without loss of generality, assume that  $\hat{x}_i^{j}\in \left(x_i^{k_j},x_i^{k_j+1}\right)$ and $\hat{x}_i^{j+1}\in\left(x_i^{k_{j+1}},x_i^{k_{j+1}+1}\right)$ for some $k_j$ and $k_{j+1}$ in $\{0,1, \cdots, \hat{m}\}$. Let $\hat{I}_{ij}=\left(x_i^{k_j},x_i^{k_j+1}\right)\cup \left(x_i^{k_j},x_i^{k_j+1}\right)$. The same proof of Lemma~\ref{3.2} leads to
\begin{eqnarray*}
  && \left\| \div{\!\!_{_\cT}}  \bff(u)-{{\text{avg}}}_{_\cT}\div \bff (u) \right\|_{L^p(K_{ij})} \\[2mm]
  &\leq & C \left\{\dfrac{h^{1/p}\delta^2}{\hat{n}^2}
\|\sigma_{tt}(x_{i},x_{i+1};\cdot)\|_{L^p(t_j,t_{j+1})}  + \dfrac{h^{2}\delta^{1/p}}{\hat{m}^2}\|u_{xx}(\cdot;t_{j},t_{j+1})\|_{L^p\big((x_i,x_{i+1})\setminus \hat{I}_{ij}\big)}\right\} \\[2mm]
&& + \dfrac{\delta}{(h\delta)^{1/q}} \sum_{l=j}^{j+1} \left|
\int^{x^{k_l+1}_i}_{x^{k_l}_i}u(x;t_{j},t_{j+1})\,dx - \hat{h}u(x^{k_l+\frac12}_{i};t_{j},t_{j+1}) \right|,
\end{eqnarray*}
which, together with Lemma~\ref{7.1}, implies
\begin{eqnarray*}
  && \left\| \div{\!\!_{_\cT}}  \bff(u)-{{\text{avg}}}_{_\cT}\div \bff (u) \right\|_{L^p(K_{ij})} \\[2mm]
&\leq & C\left(\dfrac{h^{1/p}\delta^{2}}{\hat{n}^2}
    + \dfrac{h^{2}\delta^{1/p}}{\hat{m}^2}\right) + \dfrac{\hat{h}\delta}{(h\delta)^{1/q}} \sum_{l=j}^{j+1} \left\{C\hat{h}^{1/q}\|u_{x}(\cdot;t_{j},t_{j+1})\|_{L^p\big((x_i,x_{i+1})\setminus \{\hat{x}_i^l\}\big)} +  \jump{u(\hat{x}_i^l,t_l)}\right\}.
\end{eqnarray*}
Now, (\ref{div-est2_0}) follows from $\hat{h}=h/\hat{m}$. This completes the proof of Lemma~\ref{3.3}.
\end{myproof}

\end{document}